%% file: arXiv.tex
\documentclass[hidelinks,onefignum,onetabnum]{siamart251216}

\input{shared}

\usepackage{placeins}

\ifpdf
\hypersetup{
  pdftitle={A Gradient Recovery Method for Electron Magnetohydrodynamics with Fractional Dissipation},
  pdfauthor={H. Guo, R. Hu, Q. Peng, and X. Yang}
}
\fi

\providecommand{\bU}{\bm{U}}
\providecommand{\bV}{\bm{V}}
\providecommand{\bW}{\bm{W}}
\providecommand{\Ch}{C_h}
\providecommand{\Hh}{\mathcal{H}_h}
\providecommand{\Ph}{P_h^{\mathrm{div}}}
\providecommand{\ip}[2]{\left(#1,#2\right)_h}
\providecommand{\normh}[1]{\left\|#1\right\|_{0,h}}
\providecommand{\Dx}{D_x}
\providecommand{\Dy}{D_y}

\begin{document}

\maketitle

\begin{abstract}
We propose and analyze a structure-preserving numerical method for the $2\tfrac{1}{2}$-dimensional (2.5D) electron magnetohydrodynamics system with fractional dissipation on the periodic torus. The method works directly with the magnetic field components and combines this component formulation with the gradient recovery operator of [T. Chu, H. Guo, and Z. Zhang, SIAM J. Numer. Anal., 63 (2025), pp. 23--53]. We establish discrete energy stability for a semi-implicit structure-preserving formulation and use an explicit-Hall integrating-factor implementation for efficient computation on periodic grids. The fractional dissipation is treated exactly in Fourier space, and the in-plane divergence constraint is enforced by a spectral Hodge projection. Numerical experiments demonstrate second-order spatial convergence and stable Hall-driven dynamics across several benchmark tests.
\end{abstract}



\section{Introduction}\label{sec:intro}
The electron magnetohydrodynamics (EMHD) system
\begin{equation}\label{eq:emhd-full}
  \partial_t\bB + \curl\!\bigl((\curl\bB)\times\bB\bigr) = \nu\Delta\bB,
  \qquad \Div\bB = 0,
\end{equation}
describes the evolution of a magnetic field $\bB$ in a plasma regime where electron motion dominates and ions provide a neutralizing background.  The model is widely used to study rapid magnetic reconnection and electron-scale magnetic dynamics in space and laboratory plasmas. In particular, the Hall effect allows electron motion to decouple from ion motion below the ion inertial length, leading to reconnection rates that can become largely independent of resistivity \cite{Kingsep90,Biskamp00,ShayDrake99,RogersDenton01}.

The nonlinear {Hall term} $\curl((\curl\bB)\times\bB)$ in \eqref{eq:emhd-full} contains one more spatial derivative than the nonlinear terms in standard MHD. This higher-order structure underlies many distinctive features of EMHD and makes the model considerably more challenging to analyze and discretize. On the physics side, the Hall term supports dispersive whistler waves, promotes open X-point structures in reconnection regions, and can generate elongated current sheets depending on the dominant dissipation mechanism \cite{ChaconSimakovZocco07,KnollChacon06,DaughtonsScudder06}. On the mathematical side, the non-resistive system remains poorly understood: local well-posedness in general settings is open, while recent works have shown ill-posedness mechanisms and well-posedness results under special structures or partial dissipation \cite{DaiBabaei25,JeongOh22,JeongOh24a}.

\noindent {\it The $2.5$D setting and its physical relevance.} We focus on the {$2\tfrac{1}{2}$-dimensional} ($2.5$D) setting, in which the magnetic field has three components $(B_x,B_y,B_z)$ but depends only on the two spatial variables $(x,y)$. This is the minimal nontrivial geometry for EMHD: it retains the essential Hall coupling between the in-plane field and the out-of-plane component while avoiding the full complexity of a fully three-dimensional computation. In particular, it preserves current-sheet and magnetic-island structures arising in reconnection problems \cite{KnollChacon06,ChaconSimakovZocco07,FinnKaw77,BiskampWelter80}. This reduction is therefore not merely a simplifying assumption: the Hall term still retains the derivative structure responsible for the main analytical and numerical difficulties, while the computational cost remains far below that of a fully three-dimensional simulation. Moreover, the fractional-dissipation model studied here lies in a regime where available well-posedness theory applies \cite{ChaeWanWu15}. Thus the $2.5$D system provides a natural testbed for the structure-preserving scheme developed below. 

\noindent {\it Component formulation.}
Rather than using scalar flux-function variables, we formulate the method directly in terms of the magnetic field components $(B_x,B_y,B_z)$. Under the 2.5D assumption, the curl is $\curl\bB = (\partial_y B_z,\,-\partial_x B_z,\,J)$ with $J = \partial_x B_y - \partial_y B_x$ the out-of-plane current. Expanding the Hall term component by component yields the closed system
\begin{subequations}\label{eq:emhd}
\begin{align}
  \partial_t B_x + \partial_y(\bB\cdot\nabla B_z) + \Lal B_x &= 0,
  \label{eq:emhd-x}\\
  \partial_t B_y - \partial_x(\bB\cdot\nabla B_z) + \Lal B_y &= 0,
  \label{eq:emhd-y}\\
  \partial_t B_z + \bB\cdot\nabla J + \Lal B_z &= 0,
  \label{eq:emhd-z}
\end{align}
\end{subequations}
with $\Div\bB = 0$, where $\Lal = (-\Delta)^{\alpha/2}$ is the fractional Laplacian with $\alpha\in(1,2]$. {This is the system discretized in the rest of the paper.}

The in-plane pair $(B_x,B_y)$ is driven by gradients of $B_z$; the out-of-plane component $B_z$ is advected by the in-plane current $J$. Local well-posedness of \eqref{eq:emhd} follows from Chae--Wan--Wu \cite{ChaeWanWu15}, whose analysis requires fractional dissipation of order $\alpha>1$;\footnote{In our convention $\Lal=(-\Delta)^{\alpha/2}$, equivalently $(-\Delta)^{\beta}$ with $\beta>1/2$.} our range $\alpha\in(1,2]$ lies in this regime.

\noindent {\it Relation to the scalar formulation.} The same 2.5D dynamics can also be written in scalar variables by setting $\bB = \curl(a\bm{e}_z) + b\bm{e}_z$, where $a$ is the flux function ($B_x=\partial_y a$, $B_y=-\partial_x a$) and $b=B_z$, giving 
\begin{subequations}\label{eq:dai} 
\begin{align} 
\partial_t a + a_y b_x - a_x b_y &= \mu\Delta a, \label{eq:dai-a}\\ 
\partial_t b - a_y\Delta a_x + a_x\Delta a_y &= \nu\Delta b, \label{eq:dai-b}
\end{align} 
\end{subequations} 
with current $J=-\Delta a$ and resistivities $\mu,\nu\ge0$. This form is convenient for PDE analysis, since the in-plane divergence constraint is built in and the Hall cancellation that underpins energy estimates is transparent. Local well-posedness under partial resistivity ($\mu>0$ or $\nu>0$) is due to Dai--Babaei \cite{DaiBabaei25,DaiWu23}; see \cite{DaiGlobal23,DaiIllposed24,DaiBlowup25,JeongOh22,JeongOh24a,JeongOh24b,hu2025well,HuPengYang26} for the broader literature. For low-order finite element discretization, however, \eqref{eq:dai} is less convenient than the component formulation: the Hall term in \eqref{eq:dai-b} is \emph{third order} in the flux function, through $a_x\Delta a_y - a_y\Delta a_x = \bB_{xy}\cdot\nabla J$. This leads to a parabolic time-step restriction $\tau\lesssim h^2/|\nabla a|$, far more severe than the hyperbolic $\tau\lesssim h/|\bB|$ associated with the component system, and compounds the error of approximating $J_h$ and then differentiating it again. In \eqref{eq:emhd}, by contrast, $J_h$ is computed directly from first derivatives of $(B_{x,h},B_{y,h})$, and the remaining Hall structure involves only second-order discrete operators. These considerations motivate our choice of the component formulation \eqref{eq:emhd} as the primary computational representation.

\noindent {\it Gradient recovery methods.} Gradient recovery is a classical post-processing technique in finite element analysis: from a $C^0$ finite element solution, whose gradient is generally discontinuous across element interfaces, one reconstructs a smoother and often superconvergent gradient. Two classical constructions are the superconvergent patch recovery (SPR) \cite{ZienkiewiczZhu92} and the polynomial-preserving recovery (PPR) \cite{ZhangNaga05}; recovered gradients also yield asymptotically exact a posteriori error estimators~\cite{NagaZhang04}. Recovery techniques have further been used in $C^0$ linear discretizations of fourth-order problems~\cite{GuoZhangZou18}, high-frequency wave propagation~\cite{GuoYang17}, elliptic interface problems~\cite{GuoYangInterfaceI,GuoYangInterfaceII,GuoYangInterfaceIII}, and the spectral computation of topological edge modes in photonic graphene~\cite{GuoYangZhu19}. In the present work, we employ the recovery operator \cite{CGZ25} for a different purpose: not as an error estimator, but as a constructive device that makes the second-order discrete Hall operators computable on $C^0$ linear elements, for which the required second derivatives are otherwise unavailable.

\noindent {\it Numerical approach and contributions.} The key insight is that the component formulation \eqref{eq:emhd} reduces the Hall derivative count from three to two, and the gradient recovery operator $G_h$ of Chu, Guo, and Zhang \cite{CGZ25} then provides an accurate and stable way to approximate the remaining second-order Hall structure using $C^0$ linear elements.  A semi-implicit treatment of the Hall term with an exact integrating-factor method for the fractional Laplacian gives unconditional diffusion stability, leaving only the hyperbolic CFL $\tau\lesssim h/|\bB|$ from the explicit Hall evaluation.

The present work combines the 2.5D component formulation with the gradient recovery framework \cite{CGZ25}, leading to a practical discretization based on $C^0$ linear elements and an efficient Fourier realization on periodic grids. For the analysis, we introduce a semi-implicit structure-preserving formulation and establish discrete energy stability estimates. Numerical experiments demonstrate second-order spatial convergence and illustrate stable Hall-driven dynamics across several benchmark problems.

The remainder of the paper is organized as follows. Section~\ref{sec:alg} presents the spatial discretization, the recovery gradient operator, the discrete Hall nonlinearities, and the time discretization. Section~\ref{sec:stability} establishes the discrete energy stability theory, including Hall cancellation and unconditional $L^2$ stability, and discusses a formal higher-order $H_h^s$ estimate. Section~\ref{sec:impl} describes the Fourier realization of $G_h$, the integrating-factor treatment of $\Lambda^\alpha$, and the spectral Hodge projection. Section~\ref{sec:numerics} presents the numerical experiments. Appendix~\ref{app:bdf2} introduces the BDF2 modified-energy stability theorem and its numerical verification.

\section{Algorithm Description}\label{sec:alg}
This section develops the spatial and temporal discretization of the component system~\eqref{eq:emhd}.  We first introduce a $C^0$ linear finite element space together with the gradient recovery operator that makes the second-order discrete Hall terms computable, then assemble the associated discrete Hall and current operators, and finally present the semi-implicit time discretization.  The construction keeps the lowest-order Hall cancellation exact at the discrete level, which is the structural property exploited in the stability analysis of Section~\ref{sec:stability}.

\subsection{Spatial discretization and the recovery operator}
Let $\Th$ be a shape-regular triangulation of $\Omega = [0,2\pi]^2$ with mesh size $h$.  Define the $C^0$ linear finite element space
\[
  \Vh = \{v_h\in\mathcal{P}_1(\Omega;\Th) : v_h\text{ continuous at all }
         p\in\mathcal{N}_h\},
\]
where $\mathcal{N}_h$ is the set of nodes (vertices) of $\Th$.   $V_h^{\mathrm{per}}$ denotes the subset of $\Vh$ with periodic boundary condition. For each $v_h\in\Vh$, let $\Eh$ denote the sets of all edges, and let $\me$ and $\Ke$ denote the midpoint and the union of the two triangles sharing edge $e$. Associate a Morley finite element function $w_h\in\Mh$ by
\begin{equation}\label{eq:morley-bridge}
  w_h(p) = v_h(p)\;\forall\,p\in\mathcal{N}_h,
  \qquad
  \frac{\partial w_h}{\partial\bfn_e}(\me) = \bfn_e\cdot\frac{1}{|\Ke|}\int_{\Ke}\nabla v_h\,dx \;\forall\,e\in\Eh,
\end{equation}
where $\bfn_e$ is the unit normal to $e$. 

The {gradient recovery operator} $\Gh:\Vh\to\mathcal{P}_1(\Omega;\Th)^2$ is defined by 
\[
\Gh v_h|_T = \nabla w_h|_T, \quad \text{for all} \quad T\in\Th.
\] 
We introduce $\Gh$ here on a general shape-regular triangulation; on the uniform periodic grid used in our computations it admits an explicit Fourier realization, derived in Section~\ref{sec:impl}. The key property of $\Gh$ is captured in the following lemma.

\begin{lemma}[Lemma~3.3 of \cite{CGZ25}]\label{lem:gh-cr}
  For any $v_h\in\Vh$, $\Gh v_h\in\VCR\times\VCR$, and
  \begin{equation}\label{eq:gh-formula}
    \Gh v_h(\me) = \frac{1}{|\Ke|}\int_{\Ke}\nabla v_h\,dx \qquad\forall\,e\in\Eh.
  \end{equation}
\end{lemma}

This formula makes $\Gh$ cheap to compute (one patch average per edge) and avoids assembling Morley basis functions in practice. The stability equivalence $\norm{\nabla v_h}{L^2(\Omega)}\lesssim\norm{\Gh v_h}{L^2(\Omega)}$ (Lemma~3.7 of \cite{CGZ25}) and the discrete Poincar\'e inequality $\norm{\nabla v_h}{L^2}\lesssim\norm{D\Gh v_h}{L^2(\Th)}$ hold on shape-regular meshes without uniformity assumptions.

Figure~\ref{fig:recovery} illustrates the two geometric objects underlying the definition of $\Gh$. Panel~(a) shows the diamond-shaped patch $\Ke = T_1\cup T_2$ formed by the two triangles sharing interior edge $e$: the midpoint $\me$ lies on $e$, the outward unit normal $\bfn_e$ points from $T_2$ into $T_1$, and the patch-average formula \eqref{eq:gh-formula} assigns to $\me$ the area-weighted mean of $\nabla v_h$ over the patch. Panel~(b) shows a single triangle $T$ with edge midpoints $m_{e_1}, m_{e_2}, m_{e_3}$ (squares); $\Gh v_h$ is a piecewise-linear vector field determined by its values at these midpoints, which is precisely the definition of the Crouzeix--Raviart space $\VCR\times\VCR$.

\begin{figure}[h]
  \centering
  \includegraphics[width=0.92\textwidth]{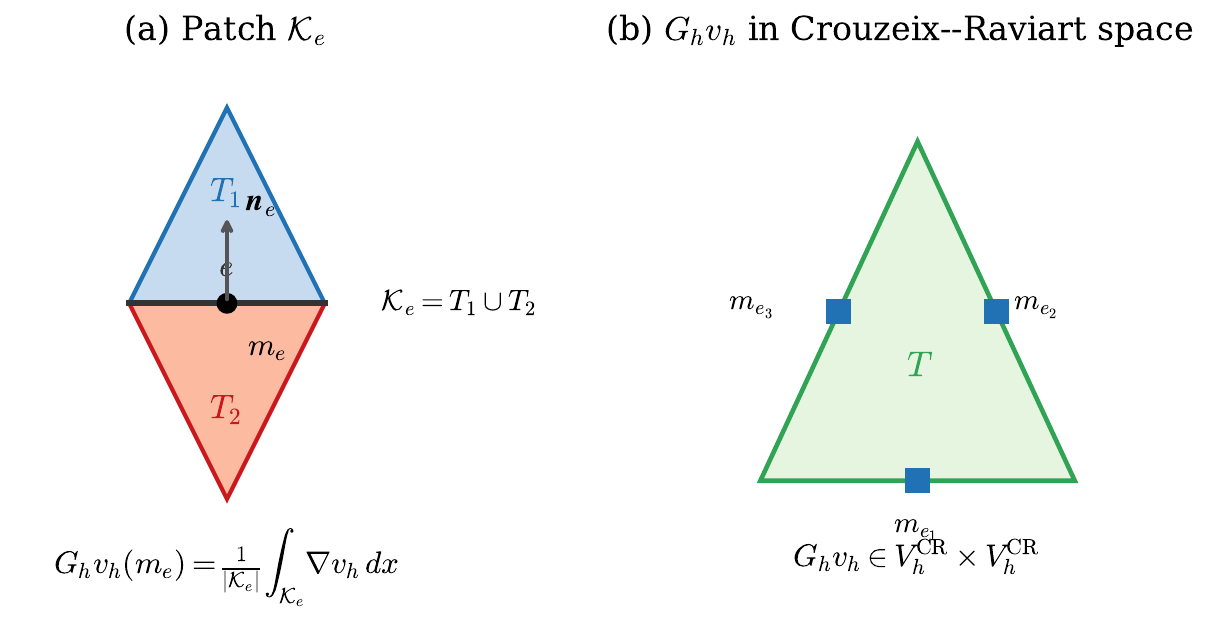}
  \caption{Recovery gradient operator $\Gh$. \emph{(a)} Patch $\Ke = T_1\cup T_2$ around interior edge $e$; the midpoint value $\Gh v_h(m_e)$ is the area-weighted average of $\nabla v_h$ over the patch (equation~\eqref{eq:gh-formula}). \emph{(b)} The recovered gradient $\Gh v_h$ evaluated at each edge midpoint (squares) belongs to the Crouzeix--Raviart space $\VCR\times\VCR$, making it a natural discrete analogue of the continuous gradient in the Sobolev space $H^1$.}
  \label{fig:recovery}
\end{figure}

\subsection{Discrete operators}
In the $2.5$D setting the discrete solution is $\bB_h = (B_{x,h},\,B_{y,h},\,B_{z,h})\in(V_h^{\mathrm{per}})^3$. The in-plane pair $(B_{x,h}, B_{y,h})$ satisfies the divergence-free constraint $\Div_{xy}\bB_h = 0$, while $B_{z,h}$ is a free scalar field.

\paragraph{1. Discrete in-plane curl}
\[
  J_h\big|_T
  \;=\;
  \bigl(\Gh B_{y,h}\cdot\bm{e}_x - \Gh B_{x,h}\cdot\bm{e}_y\bigr)\big|_T,
  \qquad T\in\Th,
\]
giving a piecewise-constant scalar $J_h = \partial_{x,h}B_{y,h} - \partial_{y,h}B_{x,h}$ (the $z$-component of $\curl_h\bB_h$).

\paragraph{2. Discrete Hall nonlinearity}
Expanding $\curl\bigl((\curl\bB)\times\bB\bigr)$ component-by-component under $\partial_z = 0$ separates the three equations as follows.
\begin{itemize}
  \item \emph{In-plane} ($i = x,y$):  the Hall force is
    $\pm\partial_i\bigl(\bB\cdot\nabla B_z\bigr)$, treated semi-implicitly:
    \begin{equation}\label{eq:hall-inplane}
      \begin{aligned}
      \mathcal{N}_h^{xy}[\bB_h^n, B_{z,h}^{n+1}]_x
        &=  \partial_{y,h}\bigl(\bB_h^n\cdot\Gh B_{z,h}^{n+1}\bigr),\\
      \mathcal{N}_h^{xy}[\bB_h^n, B_{z,h}^{n+1}]_y
        &= -\partial_{x,h}\bigl(\bB_h^n\cdot\Gh B_{z,h}^{n+1}\bigr).
      \end{aligned}
    \end{equation}
  \item \emph{Out-of-plane} ($i = z$): the Hall force is $\bB\cdot\nabla J$,
    treated semi-implicitly:
    \begin{equation}\label{eq:hall-outofplane}
      \mathcal{N}_h^{z}[\bB_h^n, \bB_h^{n+1}] = \bB_h^n\cdot\Gh J_h^{n+1},
      \qquad
      J_h^{n+1} = \partial_{x,h}B_{y,h}^{n+1} - \partial_{y,h}B_{x,h}^{n+1}.
    \end{equation}
\end{itemize}
Treating one factor of $\bB$ explicitly linearizes each Hall term and reduces the problem to three coupled second-order linear systems in $\bB_h^{n+1}$.

\paragraph{3. Discrete fractional Laplacian}
On the uniform periodic grid the fractional Laplacian $\Lal = (-\Delta)^{\alpha/2}$ is applied {exactly} in Fourier space (e.g., \cite{TangWangYuanZhou20}):
\begin{equation}\label{eq:frac-spec}
  \widehat{(\Lal f_h)}_{k_x,k_y} = |k_x^2 + k_y^2|^{\alpha/2}\,\hat{f}_{k_x,k_y}.
\end{equation}
This is exact for any $\alpha > 0$ and avoids the $O(h^2)$ stencil error of the standard five-point approximation.

\paragraph{4. Edge-penalty bilinear form}
To control jumps of the recovered gradient across edges, define (following \cite[Eq.~(4.4)]{CGZ25})
\begin{equation}\label{eq:bh}
  b_h(u_h,v_h) = \sum_{e\in\Eh} h_e^{-1}\bigl(\jump{\Gh u_h},\jump{\Gh v_h}\bigr)_{L^2(e)},
\end{equation}
where $\jump{\cdot}$ denotes the jump across $e$.

\begin{remark}
In contrast to the stabilization term in \cite{CGZ25}, the boundary penalty term is omitted here because periodic boundary conditions allow every edge to be treated as an interior edge.
\end{remark}

\subsection{Time discretization}
Given $\bB_h^n\in(V_h^{\mathrm{per}})^3$, find $\bB_h^{n+1}\in(V_h^{\mathrm{per}})^3$ such that for all $(\phi_x,\phi_y,\phi_z)\in(V_h^{\mathrm{per}})^3$:
\begin{align}
  \frac{1}{\tau}\bigl(B_{x,h}^{n+1}-B_{x,h}^n,\phi_x\bigr)
  &+\bigl(\mathcal{N}_h^{xy}[\bB_h^n,B_{z,h}^{n+1}]_x,\phi_x\bigr)
   +\bigl(\Lambda_h^\alpha B_{x,h}^{n+1},\phi_x\bigr)
   +\sigma\,b_h(B_{x,h}^{n+1},\phi_x) = 0,\label{eq:scheme-x}\\
  \frac{1}{\tau}\bigl(B_{y,h}^{n+1}-B_{y,h}^n,\phi_y\bigr)
  &+\bigl(\mathcal{N}_h^{xy}[\bB_h^n,B_{z,h}^{n+1}]_y,\phi_y\bigr)
   +\bigl(\Lambda_h^\alpha B_{y,h}^{n+1},\phi_y\bigr)
   +\sigma\,b_h(B_{y,h}^{n+1},\phi_y) = 0,\label{eq:scheme-y}\\
  \frac{1}{\tau}\bigl(B_{z,h}^{n+1}-B_{z,h}^n,\phi_z\bigr)
  &+\bigl(\mathcal{N}_h^{z}[\bB_h^n,\bB_h^{n+1}],\phi_z\bigr)
   +\bigl(\Lambda_h^\alpha B_{z,h}^{n+1},\phi_z\bigr)
   +\sigma\,b_h(B_{z,h}^{n+1},\phi_z) = 0,\label{eq:scheme-z}
\end{align}
followed by the in-plane divergence-free projection (e.g., \cite{Rossmanith06})
\begin{equation*}
  \bigl(B_{x,h}^{n+1},B_{y,h}^{n+1}\bigr)\leftarrow P_h^{\mathrm{div}}\bigl(B_{x,h}^{n+1},B_{y,h}^{n+1}\bigr).
\end{equation*}
The semi-implicit formulation leads to coupled linear subproblems at each step: the in-plane pair $(B_x, B_y)$ is driven by $\nabla B_z^{n+1}$ through \eqref{eq:scheme-x}--\eqref{eq:scheme-y}, and $B_z$ is driven by $J^{n+1}$ of $(B_x^{n+1}, B_y^{n+1})$ through \eqref{eq:scheme-z}. In the implementation of Section~\ref{sec:impl}, we use an integrating-factor variant in which all Hall forces are evaluated explicitly at time level $n$, replacing the coupled linear solve by three independent Fourier updates. The semi-implicit coupled system \eqref{eq:scheme-x}--\eqref{eq:scheme-z} is the structure-preserving model underlying the stability analysis of Section~\ref{sec:stability}; it is not solved directly in our experiments, which instead use the explicit-Hall integrating-factor implementation of Section~\ref{sec:impl}.

Figure~\ref{fig:flowchart} summarizes the time-stepping loop, combining the recovery-based Hall evaluation, the integrating-factor update, and the spectral Hodge projection.

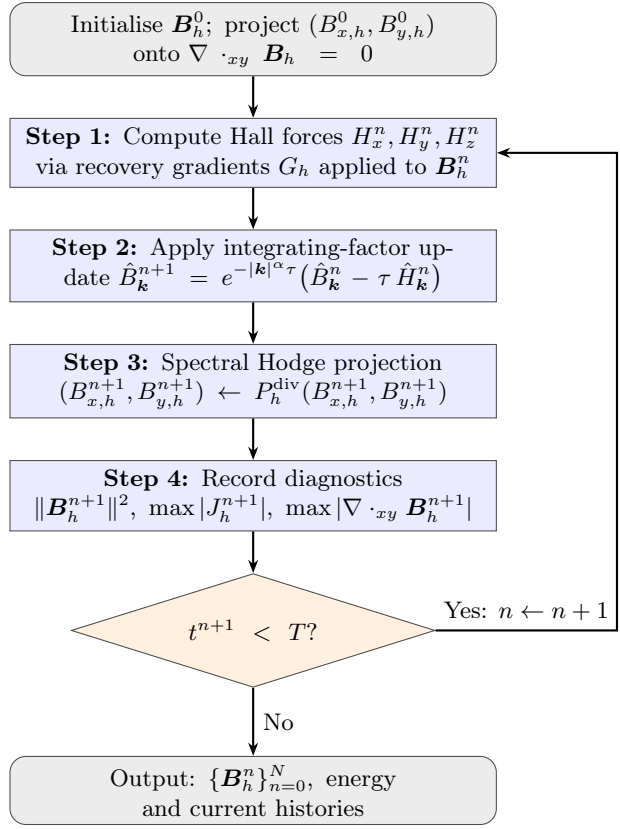
\begin{figure}[htbp]
\centering
\begin{tikzpicture}[
  startstop/.style={
    rectangle, rounded corners=6pt,
    draw=black!70, fill=gray!15,
    text width=6.2cm, align=center,
    minimum height=0.9cm, font=\small
  },
  process/.style={
    rectangle,
    draw=black!70, fill=blue!8,
    text width=6.2cm, align=center,
    minimum height=0.9cm, font=\small
  },
  decision/.style={
    diamond, aspect=3.2,
    draw=black!70, fill=orange!12,
    text width=3.6cm, align=center,
    inner sep=1pt, font=\small
  },
  arrow/.style={-{Stealth[length=5pt]}, thick},
  node distance=0.55cm
]

\node (init)   [startstop]
  {Initialise $\bB_h^0$;\enspace
   project $(B_{x,h}^0,B_{y,h}^0)$ onto $\Div_{xy}\bB_h=0$};

\node (hall)   [process, below=of init]
  {\textbf{Step 1:} Compute Hall forces
   $H_x^n, H_y^n, H_z^n$
   via recovery gradients $G_h$
   applied to $\bB_h^n$};

\node (frac)   [process, below=of hall]
  {\textbf{Step 2:} Apply integrating-factor update
   $\hat{B}^{n+1}_{\bm{k}}
    = e^{-|\bm{k}|^\alpha\tau}
      \bigl(\hat{B}^n_{\bm{k}}
            - \tau\,\hat{H}^n_{\bm{k}}\bigr)$};

\node (proj)   [process, below=of frac]
  {\textbf{Step 3:} Spectral Hodge projection
   $(B_{x,h}^{n+1},B_{y,h}^{n+1})
    \leftarrow P_h^{\mathrm{div}}(B_{x,h}^{n+1},B_{y,h}^{n+1})$};

\node (diag)   [process, below=of proj]
  {\textbf{Step 4:} Record diagnostics
   $\|\bB_h^{n+1}\|^2$,\enspace
   $\max|J_h^{n+1}|$,\enspace
   $\max|\Div_{xy}\bB_h^{n+1}|$};

\node (check)  [decision, below=of diag]
  {$t^{n+1} < T$?};

\node (output) [startstop, below=0.9cm of check]
  {Output: $\{\bB_h^n\}_{n=0}^N$,\enspace
   energy and current histories};

\draw[arrow] (init)   -- (hall);
\draw[arrow] (hall)   -- (frac);
\draw[arrow] (frac)   -- (proj);
\draw[arrow] (proj)   -- (diag);
\draw[arrow] (diag)   -- (check);
\draw[arrow] (check)  -- node[right,font=\small]{No} (output);

\draw[arrow] (check.east)
    -- node[above,font=\small]{Yes: $n\leftarrow n+1$}
    ++(2.4,0)
    |- (hall.east);

\end{tikzpicture}
\caption{Time-stepping loop for the integrating-factor EMHD scheme. Steps~1--3 are pure Fourier operations and cost $O(N^2\log N)$ per step. The Hodge projection in Step~3 maintains $\max|\Div_{xy}\bB_h|$ near round-off ($10^{-16}$--$10^{-13}$, scaling with the field amplitude) throughout.}
\label{fig:flowchart}
\end{figure}

\section{Stability analysis}\label{sec:stability}
This section presents the energy mechanism behind the scheme.  The argument follows the dyadic energy method of \cite{ChaeWanWu15}, adapted to the periodic discrete operators used here.  The essential point is that the Hall term must be discretized in a form that preserves the skew-symmetry of \(\curl((\curl\bB)\times\bB)\).  In the continuum, several componentwise forms are equivalent by the product rule and by \(\Div\bB=0\).  At the discrete level these identities are no longer automatic.  For the stability proof we therefore use the curl--cross-product form below.

Let \(\ip{\cdot}{\cdot}\) denote the periodic grid inner product and assume that \(\Dx,\Dy\) are periodic summation-by-parts operators,
\begin{equation}\label{eq:sbp}
  \ip{\Dx f}{g}=-\ip{f}{\Dx g}, \qquad \ip{\Dy f}{g}=-\ip{f}{\Dy g}.
\end{equation}
On the uniform periodic grid this is satisfied by the Fourier realization of \(\Gh\), whose symbols are purely imaginary and odd.  Define the discrete \(2.5\)D curl by
\begin{equation}\label{eq:disc-curl}
  \Ch\bB_h := \bigl(\Dy B_{z,h},\,-\Dx B_{z,h},\, \Dx B_{y,h}-\Dy B_{x,h}\bigr).
\end{equation}
Then \eqref{eq:sbp} implies the discrete curl identity
\begin{equation}\label{eq:curl-self-adjoint}
  \ip{\Ch\bU_h}{\bV_h}=\ip{\bU_h}{\Ch\bV_h},
\end{equation}
which is the periodic analogue of integration by parts for the curl operator.

For \(\bU_h,\bV_h\in(V_h^{\mathrm{per}})^3\), define the linearized Hall operator
\begin{equation}\label{eq:stable-hall-def}
  \Hh(\bU_h;\bV_h) := \Ch\bigl((\Ch\bU_h)\times\bV_h\bigr),
\end{equation}
where \(\bU_h\) is the implicit argument and \(\bV_h\) is the frozen explicit field.  This is the discrete counterpart of \(\curl((\curl\bB)\times\bB)\).

\begin{remark}\label{rem:component-conservative}
The component formula \(H_x=\Dy(\bB_{xy}\cdot\nabla_h B_z)\), \(H_y=-\Dx(\bB_{xy}\cdot\nabla_h B_z)\), \(H_z=\bB_{xy}\cdot\nabla_h J\) is equivalent to the continuum Hall term only after using the product rule and \(\Div\bB_{xy}=0\).  If this reduced component form is kept in the code, the out-of-plane term should be interpreted in conservative weak form,
\begin{equation}\label{eq:conservative-Hz}
  \begin{aligned}
  \ip{H_{z,h}(\bU_h;\bV_h)}{\phi_h} &= -\ip{J_h(\bU_h)}{V_{x,h}\Dx\phi_h+V_{y,h}\Dy\phi_h},\\
  J_h(\bU_h)&=\Dx U_{y,h}-\Dy U_{x,h},
  \end{aligned}
\end{equation}
so that the lowest-order Hall cancellation is exact.  The high-order Chae--Wan--Wu estimate is cleanest for the curl--cross-product form \eqref{eq:stable-hall-def}, and the theorem below is stated for that form.
\end{remark}

Let \(L_{b,h}\) be the symmetric non-negative operator induced by the edge penalty, namely, $  \ip{L_{b,h}u_h}{v_h}=b_h(u_h,v_h).$  The structure-preserving semi-implicit step is
\begin{equation}\label{eq:stable-scheme}
  \frac{\bB_h^{n+1,*}-\bB_h^n}{\tau} +\Hh(\bB_h^{n+1,*};\bB_h^n) +\Lambda_h^\alpha\bB_h^{n+1,*} +\sigma L_{b,h}\bB_h^{n+1,*} =0,
\end{equation}
followed by the spectral Hodge projection
\begin{align}\label{eq:stable-projection}
  \bB_h^{n+1}&=\Ph\bB_h^{n+1,*},\\
\label{eq:projection-contract}
  \normh{\Ph\bU_h}&\le \normh{\bU_h}.
\end{align}
The projection is an orthogonal projection in the grid \(L^2\) inner product; therefore

\begin{lemma}[Discrete Hall cancellation]\label{lem:disc-hall-cancel}
For every \(\bU_h,\bV_h\in(V_h^{\mathrm{per}})^3\),
\begin{equation}\label{eq:hall-cancel}
  \ip{\Hh(\bU_h;\bV_h)}{\bU_h}=0.
\end{equation}
\end{lemma}

\begin{proof}
By \eqref{eq:curl-self-adjoint} and \eqref{eq:stable-hall-def},
\begin{align*}
  \ip{\Hh(\bU_h;\bV_h)}{\bU_h}
  &=
  \ip{\Ch\bigl((\Ch\bU_h)\times\bV_h\bigr)}{\bU_h}  \\
  &=
  \ip{(\Ch\bU_h)\times\bV_h}{\Ch\bU_h}.
\end{align*}
The last inner product vanishes pointwise because \(((\Ch\bU_h)\times\bV_h)\cdot\Ch\bU_h=0\).  This proves the claim.
\end{proof}

\begin{theorem}[Discrete \(L^2\) stability]\label{thm:L2-stability}
Assume that \eqref{eq:sbp} holds, \(\Lambda_h^\alpha\) is the non-negative Fourier multiplier with symbol \(|\bm{k}|^\alpha\), and \(\sigma\ge0\).  Then one step of \eqref{eq:stable-scheme}--\eqref{eq:stable-projection} satisfies
\begin{multline}\label{eq:L2-energy-identity}
\normh{\bB_h^{n+1,*}}^2 + \normh{\bB_h^{n+1,*}-\bB_h^n}^2 +2\tau\norm{\Lambda_h^{\alpha/2}\bB_h^{n+1,*}}{0,h}^2 \\  +2\tau\sigma\sum_{i=x,y,z}b_h(B_{i,h}^{n+1,*},B_{i,h}^{n+1,*}) = \normh{\bB_h^n}^2.
\end{multline}
Consequently,
\begin{multline}\label{eq:L2-stability-final}
\normh{\bB_h^{n+1}}^2+2\tau\norm{\Lambda_h^{\alpha/2}\bB_h^{n+1,*}}{0,h}^2+2\tau\sigma\sum_{i=x,y,z}b_h(B_{i,h}^{n+1,*},B_{i,h}^{n+1,*})\\
\le\normh{\bB_h^n}^2.
\end{multline}
Thus the semi-implicit curl-form scheme is unconditionally stable in the magnetic \(L^2\) energy.
\end{theorem}

\begin{proof}
Take the grid inner product of \eqref{eq:stable-scheme} with \(\bB_h^{n+1,*}\).  The time-difference term gives
\begin{equation*}
  2\ip{\bB_h^{n+1,*}-\bB_h^n}{\bB_h^{n+1,*}}  = \normh{\bB_h^{n+1,*}}^2 -\normh{\bB_h^n}^2 +\normh{\bB_h^{n+1,*}-\bB_h^n}^2.
\end{equation*}
The Hall term vanishes by Lemma~\ref{lem:disc-hall-cancel}.  The fractional Laplacian contributes
\begin{equation*}
  \ip{\Lambda_h^\alpha\bB_h^{n+1,*}}{\bB_h^{n+1,*}}=\norm{\Lambda_h^{\alpha/2}\bB_h^{n+1,*}}{0,h}^2,
\end{equation*}
and the penalty contributes the non-negative quantity in \eqref{eq:L2-energy-identity}.  This proves \eqref{eq:L2-energy-identity}. The projected estimate \eqref{eq:L2-stability-final} follows from \eqref{eq:projection-contract}.
\end{proof}

We next state the higher-order estimate.  Let \(\Delta_q^h\), \(q\ge -1\), be the periodic discrete Littlewood--Paley blocks, truncated at the Nyquist frequency.  For \(s\in\mathbb{R}\), define
\begin{equation}\label{eq:Hs-discrete}
  \norm{\bU_h}{H_h^s}^2 := \sum_{q\ge -1}2^{2sq}\normh{\Delta_q^h\bU_h}^2.
\end{equation}
The blocks commute with \(\Ch\), \(\Lambda_h^\alpha\), and \(\Ph\) on the uniform periodic grid.

\begin{lemma}[Discrete Chae--Wan--Wu commutator estimate]
\label{lem:cww-commutator}
Let \(s>2\).  There exists a constant \(C_s\), independent of \(h\), such that for all periodic grid functions \(\bU_h,\bV_h\),
\begin{align}\label{eq:cww-commutator}
  \sum_{q\ge -1}2^{2sq} \left| \ip{\Delta_q^h\Hh(\bU_h;\bV_h)}{\Delta_q^h\bU_h} \right|  \le C_s\norm{\nabla_h\bV_h}{\ell^\infty} \norm{\bU_h}{H_h^{s+1/2}}^2.
\end{align}
\end{lemma}

\begin{proof}
Using \eqref{eq:curl-self-adjoint} and the fact that \(\Delta_q^h\) commutes with \(\Ch\),
\begin{align*}
  K_q := \ip{\Delta_q^h\Hh(\bU_h;\bV_h)}{\Delta_q^h\bU_h} = \ip{\Delta_q^h\bigl((\Ch\bU_h)\times\bV_h\bigr)}{\Delta_q^h\Ch\bU_h}.
\end{align*}
Decompose
\begin{equation*}
  \Delta_q^h\bigl((\Ch\bU_h)\times\bV_h\bigr) = (\Delta_q^h\Ch\bU_h)\times\bV_h + [\Delta_q^h,\bV_h\times]\Ch\bU_h,
\end{equation*}
where
\begin{equation*}
  [\Delta_q^h,\bV_h\times]\bW_h := \Delta_q^h(\bW_h\times\bV_h) -(\Delta_q^h\bW_h)\times\bV_h.
\end{equation*}
The leading term cancels pointwise:
\begin{equation*}
  \ip{(\Delta_q^h\Ch\bU_h)\times\bV_h}{\Delta_q^h\Ch\bU_h}=0.
\end{equation*}
Thus
\begin{equation*}
  K_q = \ip{[\Delta_q^h,\bV_h\times]\Ch\bU_h}{\Delta_q^h\Ch\bU_h}.
\end{equation*}
On the uniform periodic grid the blocks \(\Delta_q^h\) satisfy the same Bernstein and almost-orthogonality properties as their continuous counterparts, so the continuous Chae--Wan--Wu paraproduct estimate (\cite{ChaeWanWu15}) is expected to transfer to the discrete commutator with constants uniform in \(h\), giving
\begin{equation*}
  \sum_{q\ge -1}2^{2sq}|K_q| \le C_s\norm{\nabla_h\bV_h}{\ell^\infty}\norm{\bU_h}{H_h^{s+1/2}}^2.
\end{equation*}
We treat this transfer formally: a fully rigorous justification would require a discrete Littlewood--Paley calculus on the periodic grid, which we do not develop here.  The higher-order estimate of Proposition~\ref{prop:Hs-stability} should therefore be read as a formal discrete analogue of the continuous Chae--Wan--Wu result (\cite{ChaeWanWu15}), included to show that the discretization respects the same Hall cancellation mechanism at higher regularity; the unconditional $L^2$ estimate of Theorem~\ref{thm:L2-stability} is fully rigorous and does not depend on this transfer.
\end{proof}

\begin{proposition}[Discrete \(H_h^s\) stability estimate]\label{prop:Hs-stability}
Let \(1<\alpha\le2\) and \(s>2\).  Under the assumptions of Theorem~\ref{thm:L2-stability}, one step of \eqref{eq:stable-scheme}--\eqref{eq:stable-projection} satisfies
\begin{multline}\label{eq:Hs-before-absorb}
  \norm{\bB_h^{n+1,*}}{H_h^s}^2 + c\tau\norm{\bB_h^{n+1,*}}{H_h^{s+\alpha/2}}^2 + \norm{\bB_h^{n+1,*}-\bB_h^n}{H_h^s}^2     \\
  \le \norm{\bB_h^n}{H_h^s}^2 + C_s\tau \norm{\nabla_h\bB_h^n}{\ell^\infty}^{\alpha/(\alpha-1)} \norm{\bB_h^{n+1,*}}{H_h^s}^2,
\end{multline}
where \(c>0\) is independent of \(h\) and \(\tau\).  In particular, if
\begin{equation}\label{eq:time-local-condition}
  C_s\tau \norm{\nabla_h\bB_h^n}{\ell^\infty}^{\alpha/(\alpha-1)} \le \frac12,
\end{equation}
then
\begin{align}\label{eq:Hs-final}
  \norm{\bB_h^{n+1}}{H_h^s}^2 + c\tau\norm{\bB_h^{n+1,*}}{H_h^{s+\alpha/2}}^2 \le \left( 1+C_s\tau \norm{\nabla_h\bB_h^n}{\ell^\infty}^{\alpha/(\alpha-1)}\right)\norm{\bB_h^n}{H_h^s}^2.
\end{align}
Consequently, for \(m\tau\le T\),
\begin{align}\label{eq:discrete-gronwall}
  \norm{\bB_h^m}{H_h^s}^2 + c\sum_{n=0}^{m-1}\tau \norm{\bB_h^{n+1,*}}{H_h^{s+\alpha/2}}^2 \le \norm{\bB_h^0}{H_h^s}^2 \exp\Big( C_s\sum_{n=0}^{m-1}\tau \norm{\nabla_h\bB_h^n}{\ell^\infty}^{\alpha/(\alpha-1)}\Big).
\end{align}
\end{proposition}

\begin{proof}
Apply \(\Delta_q^h\) to \eqref{eq:stable-scheme}, take the inner product with \(\Delta_q^h\bB_h^{n+1,*}\), multiply by \(2^{2sq}\), and sum over \(q\). The time derivative, the fractional dissipation, and the penalty are handled as in the proof of Theorem~\ref{thm:L2-stability}.  The Hall contribution is bounded by Lemma~\ref{lem:cww-commutator} with \(\bU_h=\bB_h^{n+1,*}\) and \(\bV_h=\bB_h^n\):
\begin{equation}\label{eq:hall-Hs-bound}
  \sum_{q\ge -1}2^{2sq}|K_q| \le C_s\norm{\nabla_h\bB_h^n}{\ell^\infty} \norm{\bB_h^{n+1,*}}{H_h^{s+1/2}}^2.
\end{equation}
Since \(\alpha>1\), interpolation gives
\begin{equation}\label{eq:interp-alpha}
  \norm{\bU_h}{H_h^{s+1/2}}^2 \le \norm{\bU_h}{H_h^s}^{2(1-1/\alpha)}\norm{\bU_h}{H_h^{s+\alpha/2}}^{2/\alpha}.
\end{equation}
Combining \eqref{eq:hall-Hs-bound} and \eqref{eq:interp-alpha}, then applying
Young's inequality, yields
\begin{align*}
  C_s\norm{\nabla_h\bB_h^n}{\ell^\infty} \norm{\bB_h^{n+1,*}}{H_h^{s+1/2}}^2 &\le
  \frac{c}{2}\norm{\bB_h^{n+1,*}}{H_h^{s+\alpha/2}}^2   \\
  &\quad+ C_s\norm{\nabla_h\bB_h^n}{\ell^\infty}^{\alpha/(\alpha-1)}\norm{\bB_h^{n+1,*}}{H_h^s}^2.
\end{align*}
This proves \eqref{eq:Hs-before-absorb}.  If \eqref{eq:time-local-condition} holds, the last term can be absorbed into the left-hand side up to a harmless change of the constant \(C_s\), giving \eqref{eq:Hs-final}.  Finally, \(\Ph\) is an orthogonal Fourier multiplier and commutes with \(\Delta_q^h\), so \(\norm{\bB_h^{n+1}}{H_h^s}\le \big\|\bB_h^{n+1,*}\big\|_{H_h^s}\).  Iteration gives \eqref{eq:discrete-gronwall}.
\end{proof}

\begin{remark}
\label{rem:Hs-formal}
Proposition~\ref{prop:Hs-stability} relies on the discrete paraproduct transfer used in the proof of Lemma~\ref{lem:cww-commutator}, which we adopt formally by analogy with the continuous theory rather than prove through a discrete Littlewood--Paley calculus.  It is therefore best read as a formal discrete analogue of the Chae--Wan--Wu estimate, demonstrating that the discretization respects the same Hall cancellation mechanism at higher regularity.  The unconditional $L^2$ stability of Theorem~\ref{thm:L2-stability}, by contrast, is fully rigorous and independent of this transfer.
\end{remark}

\begin{remark}
\label{rem:alpha-threshold}
The Chae--Wan--Wu cancellation reduces the Hall contribution to a commutator that costs only one half derivative above the \(H_h^s\) energy level, namely \(H_h^{s+1/2}\).  The dissipation generated by \(\Lambda_h^\alpha=(-\Delta_h)^{\alpha/2}\) controls \(H_h^{s+\alpha/2}\).  Therefore the interpolation step \eqref{eq:interp-alpha} requires \(s+1/2<s+\alpha/2\), which is exactly \(\alpha>1\).  In the notation of \cite{ChaeWanWu15}, the diffusion is written as \((-\Delta)^\alpha\), so the corresponding condition is \(\alpha>1/2\).
\end{remark}

\begin{remark}
\label{rem:IF-stability}
The integrating-factor Euler update
\begin{equation*}
  \widehat{\bB}^{n+1}_{\bm{k}} =  e^{-|\bm{k}|^\alpha\tau} \left( \widehat{\bB}^{n}_{\bm{k}} -\tau\widehat{\mathcal{H}[\bB^n]}_{\bm{k}} \right)
\end{equation*}
is unconditionally stable for the fractional diffusion substep, but it does not inherit the exact nonlinear energy identity of Theorem~\ref{thm:L2-stability}.  Even if \(\ip{\mathcal{H}[\bB_h^n]}{\bB_h^n}=0\), the explicit Hall update gives
\begin{equation*}
  \normh{\bB_h^n-\tau\mathcal{H}[\bB_h^n]}^2 = \normh{\bB_h^n}^2 + \tau^2\normh{\mathcal{H}[\bB_h^n]}^2.
\end{equation*}
Thus the diffusion part is unconditional, whereas the fully explicit Hall part remains subject to a Hall CFL condition.  The unconditional nonlinear energy estimate is a property of the semi-implicit curl-form discretization \eqref{eq:stable-scheme}.  A second-order BDF2 variant of the scheme, which satisfies an analogous modified-energy stability theorem, is described in Appendix~\ref{app:bdf2}.
\end{remark}

\section{Implementation on the Uniform Periodic Grid}\label{sec:impl}
All computations are carried out on the uniform $N\times N$ periodic grid on $\Omega = [0,2\pi]^2$ with grid spacing $h = 2\pi/N$.  On this grid the recovery operators admit an exact Fourier realization, so all differential operators are evaluated via FFTs.  While the analysis of Section~\ref{sec:stability} is carried out for the semi-implicit curl-form discretization \eqref{eq:stable-scheme}, all numerical experiments in Section~\ref{sec:numerics} use the explicit-Hall integrating-factor implementation described below.  The two schemes are based on the same Hall operator and differential identities, but the explicit-Hall implementation is unconditionally stable only for the fractional-diffusion substep and does not inherit the exact nonlinear energy identity of Theorem~\ref{thm:L2-stability} (see Remark~\ref{rem:IF-stability}).  We use the explicit-Hall implementation in the experiments because it is computationally simpler on the periodic grid, reducing each step to a few FFTs, while the semi-implicit formulation is what exposes the discrete Hall cancellation and supplies the stability mechanism.

\paragraph{Spectral realization of the recovery gradient $G_h$}
On the uniform $N\times N$ Cartesian grid, every interior edge patch $\Ke$ consists of two right triangles whose union is an $h\times h$ square. In this geometry the patch average of $\nabla v_h$ over $\Ke$ reduces to a fixed linear combination of nodal values.  Summing over all contributing patches yields the nine-point stencil
\begin{equation}\label{eq:stencil}
  (\Gh f)^x_{i,j} = \frac{1}{4h}\Bigl[ \bigl(f_{i+1,j} - f_{i-1,j}\bigr) + \tfrac{1}{2}\bigl(f_{i+1,j+1} - f_{i-1,j+1}\bigr) + \tfrac{1}{2}\bigl(f_{i+1,j-1} - f_{i-1,j-1}\bigr) \Bigr],
\end{equation}
with the $y$-component given by symmetry. This is a separable convolution: a central difference in $x$ followed by the averaging filter $\tfrac{1}{4}[\delta_{j-1} + 2\delta_j + \delta_{j+1}]$ in $y$.

Taking the 2D DFT with integer wavenumbers $(k_x, k_y)\in\{-N/2+1,\ldots,N/2\}^2$, the Fourier symbol of $\Gh$ is
\begin{equation}\label{eq:symbol}
  \widehat{(\Gh f)^x}_{k_x,k_y} = \sigma^x(k_x, k_y)\,\hat{f}_{k_x,k_y},
  \qquad
  \sigma^x(k_x,k_y)  = -\frac{i\sin(k_x h)}{h}\cdot\frac{1+\cos(k_y h)}{2}.
\end{equation}

\begin{remark}[Accuracy and aliasing suppression]
The exact spectral $x$-derivative has symbol $ik_x$.  The stencil symbol satisfies
\[
  \sigma^x(k_x,k_y) = -ik_x\Bigl[1-\tfrac{(k_x h)^2}{6}+O((k_x h)^4)\Bigr] \Bigl[1-\tfrac{(k_y h)^2}{4}+O((k_y h)^4)\Bigr],
\]
so $\sigma^x = -ik_x + O((kh)^2)$ for resolved modes.  At the Nyquist frequency $k_y = N/2$, the averaging factor $(1+\cos(\pi))/2 = 0$ damps the corresponding near-Nyquist modes, providing a mild built-in filtering effect; this is not, however, a full de-aliasing rule for nonlinear products. In the code, \eqref{eq:stencil} is applied via \texttt{numpy.fft.fft2}/\texttt{ifft2}, which is bit-for-bit equivalent to direct stencil application but generalizes trivially to any Fourier-diagonal operator.
\end{remark}

\paragraph{Fractional Laplacian via integrating factor}
On the uniform periodic grid, $\Lal$ is applied exactly in Fourier space by \eqref{eq:frac-spec}. The integrating-factor update
\begin{equation}\label{eq:IF}
  \hat{B}^{n+1}_{\bm{k}} = e^{-|\bm{k}|^\alpha\tau} \bigl(\hat{B}^n_{\bm{k}} - \tau\,\widehat{\mathcal{H}[\bB^n]}_{\bm{k}}\bigr),
\end{equation}
where $\mathcal{H}$ denotes the explicit Hall term, handles dissipation exactly regardless of $\tau$ and $h$. The remaining practical time-step restriction is the hyperbolic Hall CFL $\tau \lesssim h/\max|\bB|$.

\begin{remark}[Unconditional stability of the IF method]
For any $\alpha > 0$ and $\tau, h > 0$, the integrating factor $e^{-|\bm{k}|^\alpha\tau} \in (0,1]$ for all $\bm{k}$. This means the diffusion step in \eqref{eq:IF} is unconditionally stable: regardless of how large $\tau$ is, no Fourier mode grows from the dissipation.  In contrast, a fully explicit treatment of the fractional diffusion would require a parabolic time-step restriction $\tau \lesssim h^\alpha / C$.
\end{remark}

\paragraph{Divergence-free projection}
After the Hall update, the in-plane components acquire a spurious divergence of order $O(\tau h)$.  This is corrected by a spectral Hodge projection: solve $\Delta\phi = \Div\bB$ in Fourier space,
\[
  \hat{\phi}_{\bm{k}} = -\frac{\widehat{(\Div\bB)}_{\bm{k}}}{|\bm{k}|^2}, \quad |\bm{k}|\neq 0,
\]
and set $\bB \leftarrow \bB - \nabla\phi$.  This reduces $\max|\Div\bB_h|$ to the round-off floor ($10^{-16}$--$10^{-13}$, scaling with the field amplitude and its gradients) and costs two forward and two inverse FFTs.  The out-of-plane component $B_z$ requires no projection.

\paragraph{Coupling summary}
The in-plane pair $(B_x, B_y)$ is forced by $\nabla B_z$ through the Hall terms \eqref{eq:hall-inplane}; $B_z$ is forced by the in-plane current $J = \partial_x B_y - \partial_y B_x$ through \eqref{eq:hall-outofplane}. In the explicit-Hall implementation, all three components use only the time-$n$ state in the Hall force evaluation, so the coupling within a step is one-directional.  The complete update is:
\begin{enumerate}[label=(\roman*),noitemsep]
  \item Evaluate Hall forces $(H_x^n, H_y^n, H_z^n)$ using \eqref{eq:symbol}.
  \item Apply integrating-factor step \eqref{eq:IF}: unconditionally stable.
  \item Project $(B_x^{n+1}, B_y^{n+1})$ via spectral Hodge: divergence to round-off.
\end{enumerate}

\section{Numerical Examples}\label{sec:numerics}
All experiments use the periodic domain $\Omega=[0,2\pi]^2$ on a uniform $N=128$ grid ($h=2\pi/128\approx0.049$) with time step $\tau=0.001$ and the integrating-factor Euler scheme of Section~\ref{sec:impl}. The Hall CFL condition $\tau \lesssim h/\max|\bB|$ is comfortably satisfied for all field amplitudes considered. The spectral Hodge projection drives $\max|\Div\bB_h|$ to the round-off floor ($10^{-16}$--$10^{-13}$ depending on the field amplitude) at every step; see Tables~\ref{tab:conv} and~\ref{tab:alpha}.

We first establish the spatial accuracy of the scheme through a convergence study, then present three qualitatively distinct benchmark initial conditions: Example~1 subjects a sharp shear layer to mild Hall forcing and tests stable evolution of a sharp-gradient configuration; Example~2 balances Hall and dissipation to exhibit island merging; and Example~3 provides a controlled verification against an analytically computable multi-mode decay formula.

\paragraph{Spatial convergence}
Table~\ref{tab:conv} reports $L^2$ and $H^1$ errors for the three fractional exponents against a fine reference solution. We use the smooth Orszag--Tang type initial condition, 
\[
\begin{aligned}
B_x &= 0.3\bigl(-\sin y + 0.1\cos(2x)\cos y\bigr),\\
B_y &= 0.3\bigl(\sin x + 0.1\sin x\,\sin(2y)\bigr),\\
B_z &= 0.15(\cos x+\cos y),
\end{aligned}
\]
followed by the spectral Hodge projection of the in-plane components. The time step $\tau$ is held {fixed} across all grid resolutions (and is also used for the reference solution). This prevents temporal errors from contaminating the comparison and isolates the observed spatial convergence, without requiring $\tau$ to scale with $h$. The reference solution is computed on a substantially finer grid, using $N=512$ for $\alpha=1.5$ and $2.0$, and $N=384$ for $\alpha=1.2$. Since weaker dissipation requires a smaller time step to maintain stability, using $N=384$ for $\alpha=1.2$ avoids the more restrictive time-step requirement associated with $N=512$.

\begin{table}[h]
\centering
\caption{Spatial convergence ($T=0.3$).  Reference at $N=512$ for $\alpha\in\{1.5,2.0\}$ ($\tau=10^{-3}$) and at $N=384$ for $\alpha=1.2$ ($\tau=5\times10^{-4}$).  Rates converge to $\approx 2$, consistent with the second-order accuracy of $G_h$ on the uniform periodic grid.  The divergence $\max|\Div\bB_h|$ stays near round-off.}
\label{tab:conv}
\smallskip
\setlength{\tabcolsep}{5pt}
\begin{tabular}{ccrrrrrr}
\toprule
$\alpha$ & $N$ & $h$ & $L^2$ error & rate & $H^1$ error & rate
         & $\max|\Div\bB_h|$ \\
\midrule
\multirow{4}{*}{$1.2$}
 & 16  & 0.3927 & $2.15\times10^{-4}$ & ---  & $4.10\times10^{-4}$ & ---  & $4.0\times10^{-16}$ \\
 & 32  & 0.1963 & $5.90\times10^{-5}$ & 1.86 & $1.29\times10^{-4}$ & 1.67 & $1.1\times10^{-15}$ \\
 & 64  & 0.0982 & $1.48\times10^{-5}$ & 1.99 & $3.35\times10^{-5}$ & 1.94 & $2.4\times10^{-15}$ \\
 & 128 & 0.0491 & $3.41\times10^{-6}$ & 2.12 & $7.77\times10^{-6}$ & 2.11 & $6.0\times10^{-15}$ \\
\midrule
\multirow{4}{*}{$1.5$}
 & 16  & 0.3927 & $1.74\times10^{-4}$ & ---  & $3.28\times10^{-4}$ & ---  & $4.3\times10^{-16}$ \\
 & 32  & 0.1963 & $4.79\times10^{-5}$ & 1.86 & $1.02\times10^{-4}$ & 1.69 & $1.1\times10^{-15}$ \\
 & 64  & 0.0982 & $1.21\times10^{-5}$ & 1.98 & $2.66\times10^{-5}$ & 1.94 & $2.6\times10^{-15}$ \\
 & 128 & 0.0491 & $2.90\times10^{-6}$ & 2.06 & $6.41\times10^{-6}$ & 2.05 & $5.5\times10^{-15}$ \\
\midrule
\multirow{4}{*}{$2.0$}
 & 16  & 0.3927 & $1.10\times10^{-4}$ & ---  & $2.02\times10^{-4}$ & ---  & $3.7\times10^{-16}$ \\
 & 32  & 0.1963 & $3.02\times10^{-5}$ & 1.86 & $6.13\times10^{-5}$ & 1.72 & $1.0\times10^{-15}$ \\
 & 64  & 0.0982 & $7.65\times10^{-6}$ & 1.98 & $1.59\times10^{-5}$ & 1.95 & $2.8\times10^{-15}$ \\
 & 128 & 0.0491 & $1.83\times10^{-6}$ & 2.06 & $3.84\times10^{-6}$ & 2.05 & $6.0\times10^{-15}$ \\
\bottomrule
\end{tabular}
\end{table}

The $L^2$ rate increases from approximately $1.86$ ($N=16\to32$) to approximately $2.1$ ($N=64\to128$), and the $H^1$ rate similarly approaches approximately $2.1$, consistent with the second-order accuracy of $G_h$ on the uniform periodic grid. The rates are consistent across all values of $\alpha$, confirming that they reflect the spatial accuracy of $G_h$ rather than $\alpha$-dependent dynamics. The divergence error remains near round-off ($10^{-16}$--$10^{-15}$) throughout, validating the spectral Hodge projection.

\noindent {\bf Example 1.} Kelvin--Helmholtz shear layer ($\alpha=1.2$).

We choose the following initial condition,
\begin{equation}\label{eq:kh-ic}
\begin{aligned}
  & B_x = -1 + \tanh\!\Bigl(\frac{y-\pi/2}{0.5}\Bigr) - \tanh\!\Bigl(\frac{y-3\pi/2}{0.5}\Bigr) + 0.1\sin x\cos 2y,\\
  & B_y = 0.1\cos 2x\sin y + 0.05\sin 3x,\\
  & B_z = 0.5\sin x\sin y + 0.3\cos 2x\cos y,
\end{aligned}
\end{equation}
followed by the spectral Hodge projection of the in-plane components. The two oppositely-signed interfaces at $y=\pi/2$ and $y=3\pi/2$, each of width $\approx 0.5$, make $B_x$ exactly periodic in $y$; this is essential on the spectral grid, since a single interface would leave a jump across the seam $y=0\equiv2\pi$ that the periodic derivative would render as a spurious boundary current dominating the physical layers.  The double-interface profile generates two genuine shear layers with $\max|J_h|\approx 2.3$ at $t=0$ (post-projection energy $E_0\approx 0.77$).  Multi-mode perturbations in $B_y$ and the non-trivial $B_z$ break translational symmetry and seed the Hall coupling.

\paragraph{Physical regime}
With $\alpha=1.2$ the fractional operator damps large wavenumbers only weakly. For the dominant modes at $|k|\sim 1$ the dissipation rate $|k|^\alpha=1$ is exceeded by the Hall forcing $|\bB||\nabla J|\approx 3$, so this is a mildly Hall-dominated test: the Hall term deforms the shear layers and couples them to $B_z$ without overwhelming the dissipation~\cite{ChaconSimakovZocco07}.

\paragraph{Dynamics}
Figure~\ref{fig:kh-snap} traces the evolution at $t\in\{0,\,0.10,\,0.25,\,0.45\}$. At $t=0$ the current density $J_h$ shows two coherent horizontal bands of opposite sign at $y=\pi/2$ and $y=3\pi/2$.  By $t=0.10$ the Hall coupling to $B_z$ has imposed a gentle in-plane undulation on these bands; $\max|J_h|$ has decreased from $2.32$ to $1.93$ and the energy from $E_0\approx 0.77$ to $\approx 0.61$.  By $t=0.25$ the bands carry a clear wavy modulation with $\max|J_h|=1.51$ ($E/E_0=0.57$), and by $t=0.45$ they remain coherent but visibly rippled, with $\max|J_h|=1.06$ ($E/E_0=0.37$).  The normalized spatial correlation between $J_h$ at $t=0$ and $t=0.45$ is $0.89$: the Hall term reshapes the layers without destroying them, so the evolution is a stable, mildly Hall-perturbed decay.  This is the expected behavior for smooth, moderate-amplitude shear data in the mildly Hall-dominated regime, and serves as a stress test of the scheme on a sharp-gradient configuration.

\paragraph{Effect of $\alpha$}
Figure~\ref{fig:alpha-comp} and Table~\ref{tab:alpha} show the sensitivity to the fractional exponent over $t\in[0,0.2]$. The normalized energy $E(T)/E_0$ equals $0.637$, $0.628$, and $0.616$ for $\alpha=1.2$, $1.5$, and $2.0$: a spread of only about $3\%$, because the energy is dominated by large-scale modes at $|k|\sim 1$ where $|k|^\alpha$ varies little between exponents. The enstrophy $\Omega(T)/\Omega_0$ equals $0.590$, $0.560$, and $0.527$: a relative spread of about $11\%$.  The asymmetry --- a small energy difference but a larger enstrophy difference --- is the fingerprint of the fractional Laplacian acting preferentially at high wavenumbers, where the $k^2$-weighted enstrophy places more of its mass.

\begin{figure}[h]
  \centering
  \includegraphics[width=\textwidth]{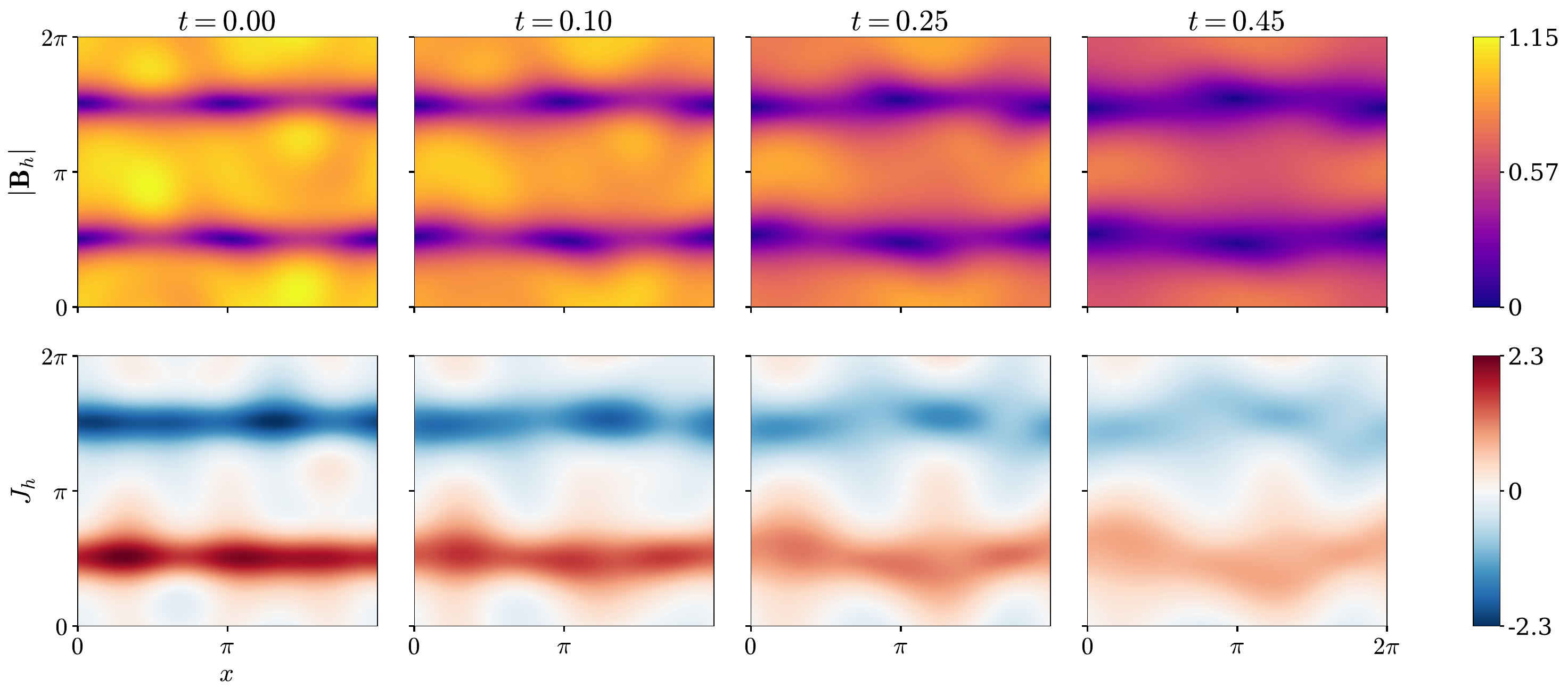}
  \caption{Kelvin--Helmholtz shear IC \eqref{eq:kh-ic} ($\alpha=1.2$, $N=128$) at $t\in\{0,\,0.10,\,0.25,\,0.45\}$. \emph{Top row}: magnetic field magnitude $|\bB_h|$ on a shared colorbar. \emph{Bottom row}: current density $J_h$ on a shared symmetric colorbar. The two opposite-signed shear layers at $y=\pi/2$ and $y=3\pi/2$ stay coherent while developing a Hall-driven undulation; $\max|J_h|$ decays from $2.3$ to $1.1$.  The spatial correlation of $J_h$ between $t=0$ and $t=0.45$ is $0.89$, confirming a stable, mildly perturbed evolution.}
  \label{fig:kh-snap}
\end{figure}

\begin{figure}[h]
  \centering
  \includegraphics[width=0.88\textwidth]{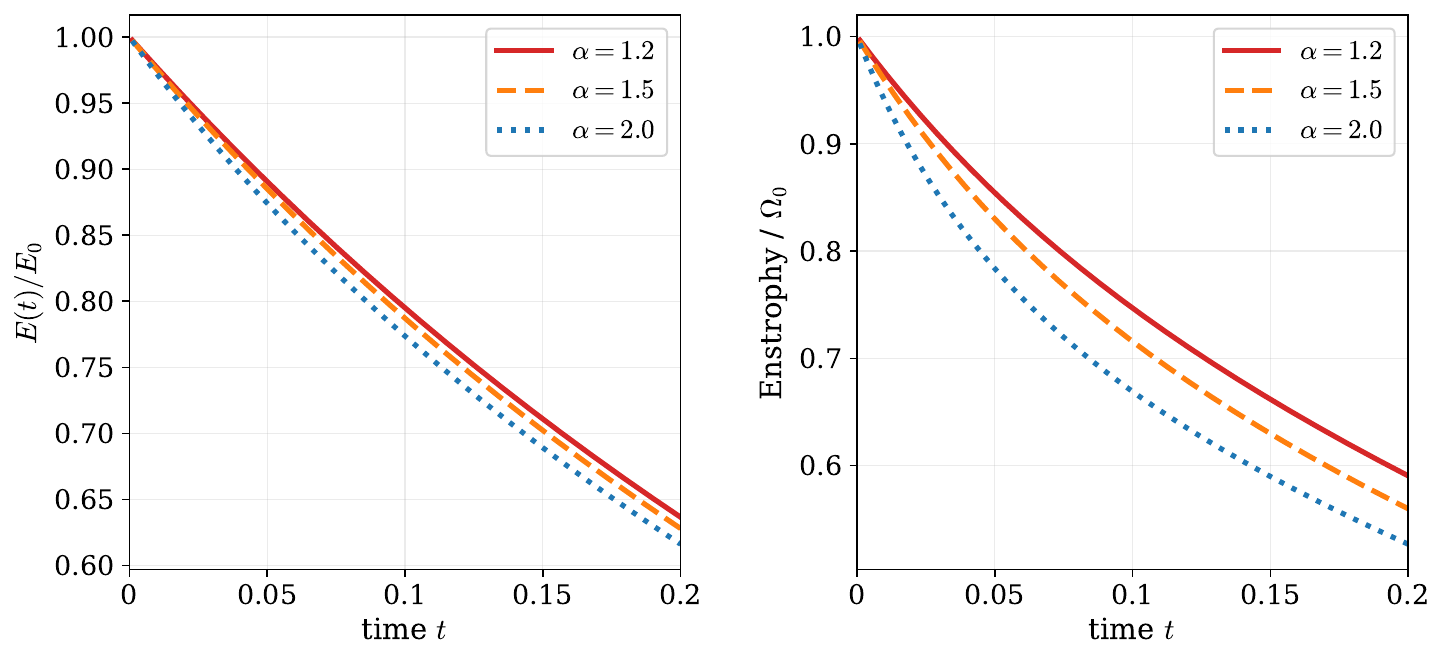}
  \caption{Effect of $\alpha\in\{1.2,\,1.5,\,2.0\}$ on the Kelvin--Helmholtz shear IC over $t\in[0,0.2]$ ($N=128$). \emph{Left}: normalized energy $E(t)/E_0$. \emph{Right}: normalized enstrophy $\Omega(t)/\Omega_0$; enstrophy is more sensitive to $\alpha$ than energy because it weights modes by $k^2$, amplifying the $|k|^\alpha$ difference at high wavenumbers. Solid, dashed, and dotted lines: $\alpha=1.2$, $1.5$, $2.0$.}
  \label{fig:alpha-comp}
\end{figure}

\begin{table}[ht]
\centering
\caption{Effect of $\alpha$ on the Kelvin--Helmholtz shear IC ($N=128$, $T=0.20$). Larger $\alpha$ accelerates both energy and enstrophy decay, with a stronger relative effect on enstrophy (which weights modes by $k^2$). Energies are normalized by the (fixed) initial value $E_0\approx0.77$; the divergence constraint is held near round-off throughout.}
\label{tab:alpha}
\smallskip
\begin{tabular}{cccc}
\toprule
$\alpha$ & $E(T)/E_0$ & $\Omega(T)/\Omega_0$ & $\max|\Div\bB_h|$ \\
\midrule
1.2 & 0.637 & 0.590 & $3.6\times10^{-14}$ \\
1.5 & 0.628 & 0.560 & $3.3\times10^{-14}$ \\
2.0 & 0.616 & 0.527 & $3.6\times10^{-14}$ \\
\bottomrule
\end{tabular}
\end{table}

\noindent {\bf Example 2.} Magnetic island merging ($\alpha=1.5$).

We choose the following initial condition,
\begin{equation}\label{eq:island-ic}
    \begin{aligned}
  B_x &= 2\bigl(\sin 2x\cos y + 0.5\sin x\cos 2y + 0.3\sin 3x\cos y\bigr),\\
  B_y &= -2\bigl(\cos 2x\sin y + 0.5\cos x\sin 2y + 0.3\cos 3x\sin y\bigr),\\
  B_z &= 2\bigl(\cos x\cos y + 0.4\cos 2x + 0.4\cos 2y + 0.2\cos 3x\cos y\bigr),
\end{aligned}
\end{equation}

followed by the spectral Hodge projection of the in-plane components. The in-plane field is a superposition of Fourier modes at wavenumbers $k=1,2,3$ that produce closed flux tubes (magnetic islands) of decreasing size; it is rendered divergence-free by the Hodge projection applied at initialization, which brings the raw energy ($\approx4.36$) down to the reported $E_0\approx 4.07$, with $\max|J_h|\approx 8.7$.  This configuration is the natural periodic analogue of the island coalescence problem studied in \cite{KnollChacon06,BiskampWelter80,FinnKaw77}.

\paragraph{Physical motivation and regime}
Magnetic island coalescence is a canonical driven-reconnection problem \cite{KnollChacon06}: the attraction between same-helicity islands drives current-sheet formation between them.  With $\alpha=1.5$ the Hall/dissipation ratio for the dominant $k=2$ modes is $|\bB||k|^2/|k|^\alpha\approx 2.8$, so the Hall term and the dissipation are in genuine competition, deforming the islands while smoothing fine scales.

\paragraph{Dynamics}
Figure~\ref{fig:island-snap} traces the evolution at $t\in\{0,\,0.10,\,0.25,\,0.40\}$. At $t=0$ the $J_h$ field shows regular thin current filaments at the island boundaries. By $t=0.10$ the Hall term has distorted the boundaries and generated fine-scale current concentrations; the smallest ($k=3$) islands begin merging with neighbors. By $t=0.25$ the $k=3$ islands have been substantially absorbed, and the current pattern is dominated by elongated filaments between formerly distinct islands, showing structures consistent with Hall-dominated reconnection. By $t=0.40$ the multi-island pattern has reorganized into two or three dominant flux-tube structures; the $J_h$ correlation between $t=0$ and $t=0.40$ is $0.51$, indicating significant structural change while some memory of the initial configuration is retained.

\begin{figure}[h]
  \centering
  \includegraphics[width=\textwidth]{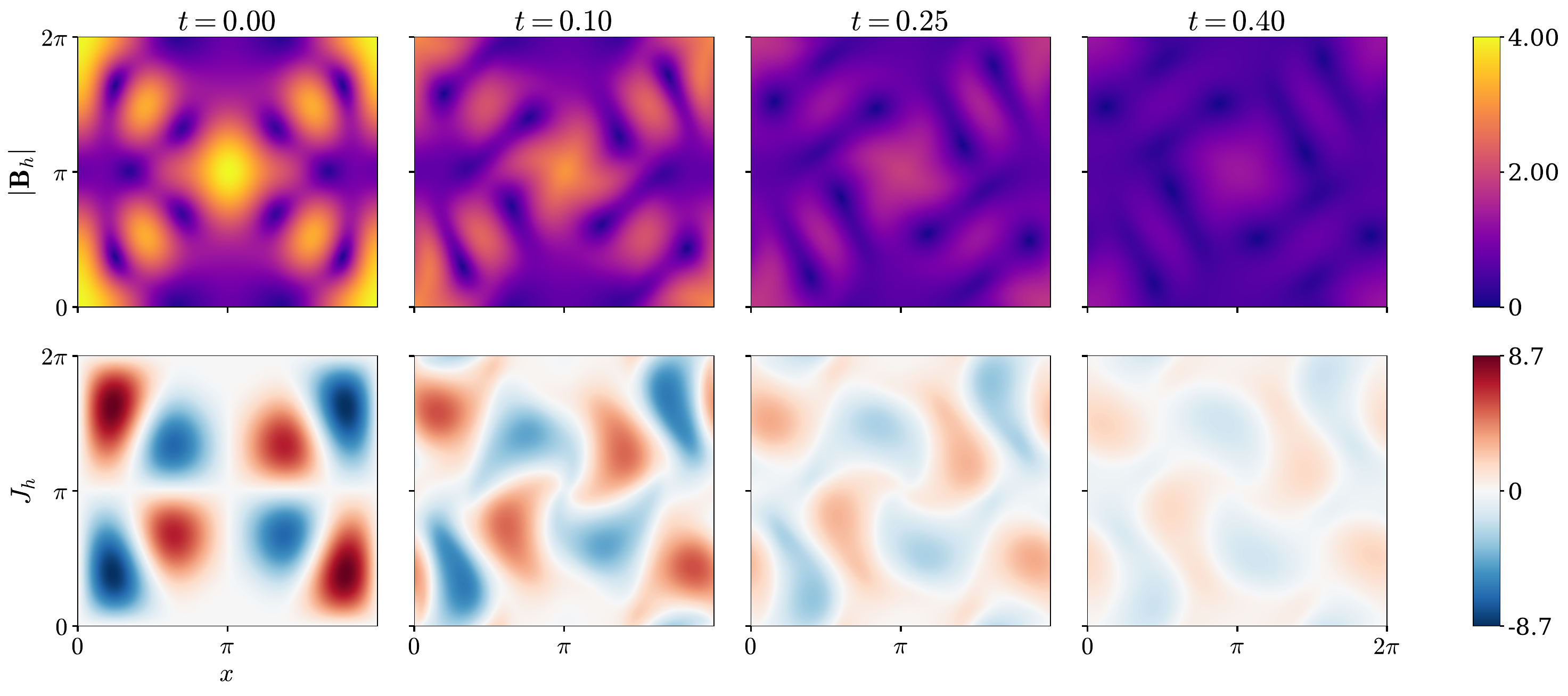}
  \caption{Magnetic island IC \eqref{eq:island-ic} ($\alpha=1.5$, $N=128$) at $t\in\{0,\,0.10,\,0.25,\,0.40\}$. \emph{Top row}: $|\bB_h|$ with per-panel colorbar; the multi-island structure progressively coarsens as Hall-driven merging and fractional dissipation reduce the island count. \emph{Bottom row}: $J_h$ with per-panel symmetric colorbar; the initially regular current filaments at island boundaries develop irregular fine-scale structure and partially merge. Spatial $J_h$ correlation falls from $1$ to $0.51$ over this window.}
  \label{fig:island-snap}
\end{figure}

\noindent {\bf Example 3.} Helical initial data and verification ($\alpha=2.0$).

We choose the following initial condition,
\begin{equation}\label{eq:helical-ic}
  \begin{aligned}
  B_x &= \sin x\cos y + 0.5\cos(2x)\sin y,\\
  B_y &= -\cos x\sin y + 0.5\sin x\cos(2y),\\
  B_z &= \sin x + \sin y.
  \end{aligned}
\end{equation}
After the divergence-free projection, the initial energy of \eqref{eq:helical-ic} is $E_0\approx1.53$. Under $\Lambda^2=-\Delta$,     the Fourier modes with $|k|^2=1$, $2$, and $5$ decay respectively as $e^{-t}$, $e^{-2t}$, and $e^{-5t}$, and account for approximately   $65.6\%$, $32.8\%$, and $1.6\%$ of the total energy. These percentages are the post-projection energy split: the $|k|^2=5$ in-plane pair in \eqref{eq:helical-ic} is not divergence-free, so the Hodge projection removes most of its energy and lowers $E_0$ from its raw value $\approx 1.63$ to $\approx 1.53$. Since $E(t) = \|\bB\|_{L^2}^2$ is quadratic, the linearized energy follows
\begin{equation}\label{eq:helical-exact}
  \frac{E(t)}{E_0} = 0.656\,e^{-2t} + 0.328\,e^{-4t} + 0.016\,e^{-10t} + \mathcal{O}(\|\bB\|^2).
\end{equation}

\paragraph{Verification}
With $\alpha=2$ the integrating factor is $e^{-|k|^2\tau}$, exact with no approximation.  Any deviation from \eqref{eq:helical-exact} measures the cumulative effect of the Hall nonlinearity. The Hall/dissipation ratio for Group~B modes is $|\bB||k|^2/|k|^\alpha = 1\cdot 2/2 = 1$: Hall and dissipation are comparable, making this a mildly nonlinear verification test rather than a turbulence simulation.

Table~\ref{tab:helical-verify} compares $E(t)/E_0$ against \eqref{eq:helical-exact} at six output times. The error is $0.60\%$ at $t=0.25$, growing to $2.88\%$ at $t=1.0$ and then saturating near $3\%$ for $t\geq 1.0$.  The plateau indicates that the nonlinear correction is accumulated during the active Hall phase $t\in[0,1]$ and then freezes as the field becomes too small for the Hall term to matter. A $3\%$ departure from the linearized formula at $T=1$ confirms that the Hall term is present and active but sub-dominant, as expected for moderate-amplitude smooth initial data.

\begin{table}[h]
\centering
\caption{Semi-analytic verification for the helical IC \eqref{eq:helical-ic} with $\alpha=2.0$, $N=128$, $\tau=0.001$.  The exact formula \eqref{eq:helical-exact} gives the linearized (Hall-free) energy. The percentage error measures the cumulative Hall nonlinearity effect and saturates near $3\%$, confirming Hall is active but sub-dominant.}
\label{tab:helical-verify}
\smallskip
\begin{tabular}{ccccc}
\toprule
$t$ & $E(t)/E_0$ (simulation) & $E(t)/E_0$ (exact) & error & $\max|J_h|$ \\
\midrule
0.25 & 0.51676 & 0.51986 & 0.60\% & 1.260 \\
0.50 & 0.28071 & 0.28583 & 1.79\% & 0.749 \\
0.75 & 0.15854 & 0.16271 & 2.56\% & 0.449 \\
1.00 & 0.09205 & 0.09479 & 2.88\% & 0.271 \\
1.50 & 0.03248 & 0.03347 & 2.97\% & 0.098 \\
2.00 & 0.01177 & 0.01213 & 2.91\% & 0.035 \\
\bottomrule
\end{tabular}
\end{table}

\paragraph{Energy and current histories}
Figure~\ref{fig:energy} compares all three examples over $t\in[0,0.6]$. The left panel plots absolute energy $E(t)$ on a log scale. The three examples have initial amplitudes $E_0\approx 0.77$ (KH shear), $4.07$ (magnetic islands), and $1.53$ (helical), and decay at rates determined jointly by $\alpha$ and spatial structure. The right panel plots $\max|J_h(t)|$, a sensitive indicator of current-sheet intensity.  The magnetic-island example carries the highest current density throughout, reflecting its larger field amplitude and the thin current filaments at the island boundaries; the helical example decays most smoothly, consistent with the linearized formula \eqref{eq:helical-exact}.

\begin{figure}[h]
  \centering
  \includegraphics[width=0.99\textwidth]{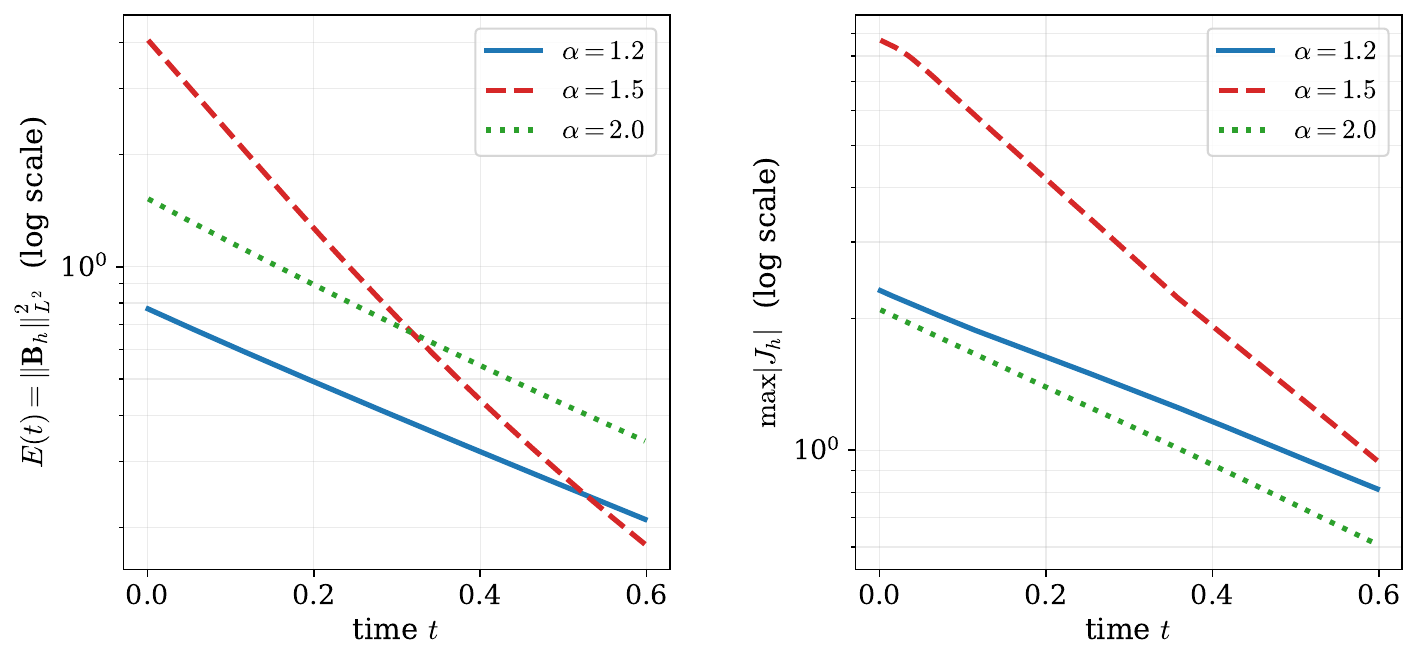}
  \caption{Comparison of all examples over $t\in[0,0.6]$ ($N=128$). Solid: KH shear ($\alpha=1.2$); dashed: magnetic islands ($\alpha=1.5$); dotted: helical ($\alpha=2.0$). \emph{Left}: absolute energy $E(t)$ on log scale, showing different initial amplitudes and $\alpha$-dependent decay rates. \emph{Right}: peak current $\max|J_h(t)|$ on log scale; the magnetic-island sustains the highest current intensity, while the helical follows the smooth linearized decay.}
  \label{fig:energy}
\end{figure}

\section*{Acknowledgments}
This work was partially supported by the ONR grant under \#N00014-24-1-2432, the Simons Foundation (MP-TSM-00002783) and the NSF grant DMS-2420988. During the preparation of this manuscript, the authors used ChatGPT for language polishing and to improve the clarity of the presentation. The authors reviewed and edited the generated content as needed and assume full responsibility for all content of the manuscript.

\appendix
\FloatBarrier
\section{Stability of a BDF2 variant}\label{app:bdf2}
The energy framework of Section~\ref{sec:stability} extends to a second-order-in-time variant based on the Backward Differentiation Formula of order 2 (BDF2).  The numerical experiments are based primarily on the integrating-factor Euler implementation; the results below provide mathematical support for future use of higher-order time stepping.

For problems where a larger time step is desirable, the energy framework extends naturally to the second-order Backward Differentiation Formula (BDF2) scheme.  Given $\bB_h^n$ and $\bB_h^{n-1}$, let $\widetilde{\bB}_h^n := 2\bB_h^n - \bB_h^{n-1}$ be the second-order extrapolation.  The BDF2 semi-implicit step is
\begin{equation}\label{eq:bdf2-scheme}
  \frac{3\bB_h^{n+1,*}-4\bB_h^n+\bB_h^{n-1}}{2\tau} +\Hh(\bB_h^{n+1,*};\widetilde{\bB}_h^n) +\Lambda_h^\alpha\bB_h^{n+1,*} +\sigma L_{b,h}\bB_h^{n+1,*} =0,
\end{equation}
followed by the Hodge projection $\bB_h^{n+1}=\Ph\bB_h^{n+1,*}$.

Define the BDF2 modified magnetic energy
\begin{equation}\label{eq:bdf2-energy}
  \mathcal{E}_h^n := \frac{1}{2}\Bigl( \normh{\bB_h^n}^2 +\normh{2\bB_h^n-\bB_h^{n-1}}^2 \Bigr).
\end{equation}
The factor $\normh{2\bB_h^n-\bB_h^{n-1}}^2$ is the standard BDF2 correction term that makes the modified energy equivalent to $\normh{\bB_h^n}^2$ up to a constant factor.

\begin{theorem}[BDF2 modified-energy stability]\label{thm:bdf2}
Under the assumptions of Theorem~\ref{thm:L2-stability}, if $\bB_h^n$ and $\bB_h^{n-1}$ are in the range of $\Ph$, then every step of \eqref{eq:bdf2-scheme} satisfies
\begin{align}\label{eq:bdf2-energy-identity}
  \mathcal{E}_h^{n+1,*} &+ \frac{1}{2}\normh{\bB_h^{n+1,*}-2\bB_h^n+\bB_h^{n-1}}^2 +2\tau\normh{\Lambda_h^{\alpha/2}\bB_h^{n+1,*}}^2 \notag\\
  &\quad+ 2\tau\sigma\sum_{i=x,y,z}b_h(B_{i,h}^{n+1,*},B_{i,h}^{n+1,*}) = \mathcal{E}_h^{n},
\end{align}
where $\mathcal{E}_h^{n+1,*}$ is defined with $\bB_h^{n+1,*}$ in place of $\bB_h^{n+1}$.  Consequently,
\begin{equation}\label{eq:bdf2-stability-final}
  \mathcal{E}_h^{n+1} +2\tau\normh{\Lambda_h^{\alpha/2}\bB_h^{n+1,*}}^2 +2\tau\sigma\sum_{i=x,y,z}b_h(B_{i,h}^{n+1,*},B_{i,h}^{n+1,*}) \le \mathcal{E}_h^{n}.
\end{equation}
\end{theorem}

\begin{proof}
Take the inner product of \eqref{eq:bdf2-scheme} with $\bB_h^{n+1,*}$. The Hall term vanishes by Lemma~\ref{lem:disc-hall-cancel} (applied with the frozen field $\widetilde{\bB}_h^n$ in place of $\bV_h$).  The BDF2 time-difference identity
\begin{align*}
  2\ip{3a-4b+c}{a} = \normh{a}^2+\normh{2a-b}^2 -\normh{b}^2-\normh{2b-c}^2 +\normh{a-2b+c}^2
\end{align*}
with $a=\bB_h^{n+1,*}$, $b=\bB_h^n$, $c=\bB_h^{n-1}$, multiplied by $2\tau$, gives \eqref{eq:bdf2-energy-identity}. The projected estimate \eqref{eq:bdf2-stability-final} follows because $\Ph$ is orthogonal and $\bB_h^n$ is already in its range, so $\normh{\bB_h^{n+1}} \le \normh{\bB_h^{n+1,*}}$ and $\normh{2\bB_h^{n+1}-\bB_h^n} \le \normh{2\bB_h^{n+1,*}-\bB_h^n}$.
\end{proof}

\begin{remark}[Comparison of the two schemes]
Both the backward-Euler scheme \eqref{eq:stable-scheme} and the BDF2 scheme \eqref{eq:bdf2-scheme} are unconditionally stable in their respective energy norms, and both treat the Hall nonlinearity with a single frozen explicit factor, so both require the Hall CFL $\tau\lesssim h/\max|\bB_h|$ for accuracy rather than stability.  The backward-Euler scheme is first-order in time and simpler to start (no previous-step value required).  The BDF2 scheme carries a larger stability constant and requires a warm-start Euler step, but its modified-energy identity \eqref{eq:bdf2-energy-identity} retains the same unconditional character.  In the numerical experiments of Section~\ref{sec:numerics} we use the integrating-factor Euler implementation as the primary scheme; Table~\ref{tab:bdf2} records the modified-energy behavior of the BDF2 variant for completeness.
\end{remark}

\paragraph{Numerical verification}
Table~\ref{tab:bdf2} verifies Theorem~\ref{thm:bdf2} numerically. For each of the three examples and each $\alpha\in\{1.2,1.5,2.0\}$, we run $200$ BDF2 steps on an $N=32$ grid with $\tau=0.002$, and record the initial and final modified energy $\mathcal{E}_h^n$ defined in \eqref{eq:bdf2-energy}, together with the number of steps at which $\mathcal{E}_h^n$ increases (a violation of Theorem~\ref{thm:bdf2}).

In every case the modified energy is strictly monotone throughout the run, confirming the theoretical estimate \eqref{eq:bdf2-stability-final}. The ratio $\mathcal{E}_h^{200}/\mathcal{E}_h^0$ reflects physical energy dissipation: faster fractional damping (larger $\alpha$) leads to a smaller final-to-initial ratio.

\begin{table}[htbp]
\centering
\caption{BDF2 modified-energy $\mathcal{E}_h^n$ verification ($N=32$, $\tau=0.002$, $200$ steps). The column ``non-mono'' counts steps at which $\mathcal{E}_h^n$ increases; zero in every case confirms Theorem~\ref{thm:bdf2}. Ratios reflect physical dissipation; larger $\alpha$ dissipates more rapidly.}
\label{tab:bdf2}
\smallskip
\setlength{\tabcolsep}{6pt}
\begin{tabular}{llrrrr}
\toprule
Examples & $\alpha$ & $\mathcal{E}_h^0$ & $\mathcal{E}_h^{200}$
                  & $\mathcal{E}_h^{200}/\mathcal{E}_h^0$ & non-mono \\
\midrule
\multirow{3}{*}{KH shear}
 & 1.2 & 0.7672 & 0.3175 & 0.4139 & 0 \\
 & 1.5 & 0.7668 & 0.3131 & 0.4083 & 0 \\
 & 2.0 & 0.7659 & 0.3080 & 0.4021 & 0 \\
\midrule
\multirow{3}{*}{Magnetic islands}
 & 1.2 & 4.0172 & 0.6261 & 0.1558 & 0 \\
 & 1.5 & 4.0029 & 0.4464 & 0.1115 & 0 \\
 & 2.0 & 3.9701 & 0.2605 & 0.0656 & 0 \\
\midrule
\multirow{3}{*}{Helical}
 & 1.2 & 1.5141 & 0.5942 & 0.3924 & 0 \\
 & 1.5 & 1.5135 & 0.5737 & 0.3790 & 0 \\
 & 2.0 & 1.5123 & 0.5431 & 0.3591 & 0 \\
\bottomrule
\end{tabular}
\end{table}


\FloatBarrier

\bibliographystyle{siamplain}
\bibliography{references}
\end{document}

%% file: shared.tex
\usepackage{amsfonts}
\usepackage{amssymb}
\usepackage{graphicx}
\usepackage{epstopdf}
\usepackage{algorithmic}
 
\newsiamremark{remark}{Remark}
\newsiamremark{hypothesis}{Hypothesis}
\crefname{hypothesis}{Hypothesis}{Hypotheses}
\newsiamthm{claim}{Claim}
\newsiamremark{fact}{Fact}
\crefname{fact}{Fact}{Facts}

\usepackage[T1]{fontenc}
\usepackage[utf8]{inputenc}
\usepackage{mathtools}
\usepackage{bm}
\usepackage{booktabs}
\usepackage{multirow}
\usepackage{array}
\usepackage{xcolor}
\usepackage{graphicx}
\usepackage{stmaryrd}
\usepackage{caption}
\usepackage{subcaption}
\usepackage{enumitem}
\usepackage{tikz}
\usetikzlibrary{shapes.geometric,arrows.meta,positioning,calc}
\usepackage{mdframed}
\usepackage{parskip}
\usepackage[numbers,sort&compress]{natbib}
 
\newcommand{\bB}{\bm{B}}

\newcommand{\Gh}{G_h}
\newcommand{\Th}{\mathcal{T}_h}
\newcommand{\Vh}{V_h}

\newcommand{\VCR}{V_h^{\mathrm{CR}}}
\newcommand{\Mh}{M_h}
\newcommand{\Eh}{\mathcal{E}_h}

\newcommand{\Ke}{\mathcal{K}_e}
\newcommand{\me}{m_e}
\newcommand{\bfn}{\bm{n}}

\newcommand{\Lal}{\Lambda^\alpha}
\newcommand{\jump}[1]{\llbracket #1 \rrbracket}
\newcommand{\norm}[2]{\left\|#1\right\|_{#2}}

\newcommand{\curl}{\nabla\times}
\newcommand{\Div}{\nabla\cdot}

\newcommand{\normh}[1]{\left\| #1 \right\|_{0,h}}

\newcommand{\Dx}{D_x}
\newcommand{\Dy}{D_y}

\newcommand{\ip}[2]{\left( #1,#2 \right)_h}

\graphicspath{{fig/}}

\headers{A gradient recovery method for EMHD}{H. Guo, R. Hu, Q. Peng, and X. Yang}

\title{A Gradient Recovery Method for Electron Magnetohydrodynamics with Fractional Dissipation}
\author{Hailong Guo\thanks{  Chongqing Research Institute of Big Data, Peking University, Chongqing, 401332, China  and  School of Mathematics and Statistics, The University of Melbourne, Parkville, VIC 3010, Australia.
  \texttt{hailong.guo@cqbdri.pku.edu.cn, hailong.guo@unimelb.edu.au} }
\and
Ruimeng Hu\thanks{Department of Mathematics, Department of Statistics and Applied Probability,
  University of California, Santa Barbara, CA 93106, USA.
  \texttt{rhu@ucsb.edu}}
\and
Qirui Peng\thanks{Department of Mathematics,
  University of California, Santa Barbara, CA 93106, USA.
  \texttt{qpeng9@ucsb.edu}}
\and
Xu Yang\thanks{Department of Mathematics,
  University of California, Santa Barbara, CA 93106, USA.
  \texttt{xy6@ucsb.edu}}}

\usepackage{amsopn}


%% file: references.bib
@article{DaiBabaei25,
  author    = {Dai, Mimi and Babaei, Hassan},
  title     = {Well-posedness of the electron {MHD} with partial resistivity},
  journal   = { arXiv:2503.18149},
  year      = {2025},
}

@article{DaiGlobal23,
  author    = {Dai, Mimi},
  title     = {Global existence of {2D} electron {MHD} near a steady state},
  journal   = { arXiv:2306.13036},
  year      = {2023},
}

@article{DaiIllposed24,
  author    = {Dai, Mimi},
  title     = {Ill-posedness of {$2\tfrac{1}{2}$D} electron {MHD}},
  journal   = { arXiv:2411.00120},
  year      = {2024},
}

@article{DaiBlowup25,
  author    = {Dai, Mimi},
  title     = {Blowup for the forced electron {MHD}},
  journal   = { arXiv:2503.14777},
  year      = {2025},
}

@article{DaiWu23,
  author    = {Dai, Mimi and Wu, Cheng},
  title     = {Dissipation wavenumber and regularity for electron magnetohydrodynamics},
  journal   = {Journal of Differential Equations},
  volume    = {376},
  pages     = {655--681},
  year      = {2023},
}

@article{JeongOh22,
  author    = {Jeong, In-Jee and Oh, Sung-Jin},
  title     = {On the {Cauchy} problem for the {Hall} and electron magnetohydrodynamic equations without resistivity {I}: {I}llposedness near degenerate stationary solutions},
  journal   = {Annals of PDE},
  volume    = {8},
  number    = {15},
  year      = {2022},
}

@article{JeongOh24a,
  author    = {Jeong, In-Jee and Oh, Sung-Jin},
  title     = {On illposedness of the {Hall} and electron magnetohydrodynamic equations without resistivity on the whole space},
  journal   = { arXiv:2404.13790},
  year      = {2024},
}

@article{JeongOh24b,
  title={Wellposedness of the Electron MHD Without Resistivity for Large Perturbations of the Uniform Magnetic Field},
  author={Jeong, In-Jee and Oh, Sung-Jin},
  journal={Annals of PDE},
  volume={11},
  number={1},
  pages={14},
  year={2025}
}

@article{HuPengYang26,
  author    = {Hu, Ruimeng and Peng, Qirui and Yang, Xu},
  title     = {The three-dimensional stochastic {EMHD} system: {L}ocal well-posedness and maximal pathwise solutions},
  journal   = { arXiv:2604.07497},
  year      = {2026},
}

@article{ChaeWanWu15,
  author    = {Chae, Dongho and Wan, Renhui and Wu, Jiahong},
  title     = {Local well-posedness for the {Hall--MHD} equations with fractional magnetic diffusion},
  journal   = {Journal of Mathematical Fluid Mechanics},
  volume    = {17},
  number    = {4},
  pages     = {627--638},
  year      = {2015}
}

@article{CGZ25,
  author    = {Chu, Tianhao and Guo, Hailong and Zhang, Zhimin},
  title     = {Recovery based linear finite element methods for {Hamilton--Jacobi--Bellman} equation with {Cordes} coefficients},
  journal   = {SIAM Journal on Numerical Analysis},
  volume    = {63},
  number    = {1},
  pages     = {23--53},
  year      = {2025},
}

@article{ChaconSimakovZocco07,
  author    = {Chac\'{o}n, L. and Simakov, Andrei N. and Zocco, A.},
  title     = {Steady-state properties of driven magnetic reconnection in {2D} electron magnetohydrodynamics},
  journal   = {Physical Review Letters},
  volume    = {99},
  pages     = {235001},
  year      = {2007},
}

@article{KnollChacon06,
  author    = {Knoll, D. A. and Chac\'{o}n, L.},
  title     = {Coalescence of magnetic islands in the low-resistivity, {Hall--MHD} regime},
  journal   = {Physical Review Letters},
  volume    = {96},
  pages     = {135001},
  year      = {2006},
}

@article{Kingsep90,
  title={Electron magnetohydrodynamics},
  author={Gordeev, Aleksandr V. and Kingsep, Aleksandr S. and Rudakov, Leonid I.},
  journal={Physics Reports},
  volume={243},
  number={5},
  pages={215--315},
  year={1994},
  publisher={Elsevier}
}

@article{Biskamp00,
  title={Magnetic reconnection in plasmas},
  author={Biskamp, Dieter},
  journal={Astrophysics and Space Science},
  volume={242},
  number={1},
  pages={165--207},
  year={1996}
}

@article{ShayDrake99,
  author    = {Shay, M. A. and Drake, J. F. and Rogers, B. N. and Denton, R. E.},
  title     = {The scaling of collisionless, magnetic reconnection for large systems},
  journal   = {Geophysical Research Letters},
  volume    = {26},
  pages     = {2163--2166},
  year      = {1999},
}

@article{RogersDenton01,
  author    = {Rogers, B. N. and Denton, R. E. and Drake, J. F. and Shay, M. A.},
  title     = {Role of dispersive waves in collisionless magnetic reconnection},
  journal   = {Physical Review Letters},
  volume    = {87},
  pages     = {195004},
  year      = {2001},
}

@article{DaughtonsScudder06,
  author    = {Daughton, W. and Scudder, J. and Karimabadi, H.},
  title     = {Fully kinetic simulations of undriven magnetic reconnection with open boundary conditions},
  journal   = {Physics of Plasmas},
  volume    = {13},
  pages     = {072101},
  year      = {2006},
}

@techreport{FinnKaw77,
  title={Coalescence instability of magnetic islands},
  author={Finn, J Michael and Kaw, PK},
  year={1976},
  institution={Princeton Univ., NJ (USA). Plasma Physics Lab.}
}

@article{BiskampWelter80,
  author    = {Biskamp, D. and Welter, H.},
  title     = {Coalescence of magnetic islands},
  journal   = {Physical Review Letters},
  volume    = {44},
  pages     = {1069--1072},
  year      = {1980},
}

@article{GuoYangInterfaceI,
  author = {Guo, Hailong and Yang, Xu},
  title = {Gradient recovery for elliptic interface problem: {I}. body-fitted mesh},
  journal = {Commun. Comput. Phys.},
  volume = {23}, 
  pages = {1488--1511}, 
  year = {2018}
  }

@article{GuoYangInterfaceII,
  title={Gradient recovery for elliptic interface problem: {II}. Immersed finite element methods},
  author={Guo, Hailong and Yang, Xu},
  journal={Journal of Computational Physics},
  volume={338},
  pages={606--619},
  year={2017}
}

@article{GuoYangInterfaceIII,
  title={Gradient recovery for elliptic interface problem: {III}. {N}itsche's method},
  author={Guo, Hailong and Yang, Xu},
  journal={Journal of Computational Physics},
  volume={356},
  pages={46--63},
  year={2018},
  publisher={Elsevier}
}

@article{GuoYangZhu19,
  title={Bloch theory-based gradient recovery method for computing topological edge modes in photonic graphene},
  author={Guo, Hailong and Yang, Xu and Zhu, Yi},
  journal={Journal of Computational Physics},
  volume={379},
  pages={403--420},
  year={2019},
  publisher={Elsevier}
}

@article{ZienkiewiczZhu92,
  title={The superconvergent patch recovery and a posteriori error estimates. {P}art 1: {T}he recovery technique},
  author={Zienkiewicz, Olgierd Cecil and Zhu, Jian Zhong},
  journal={International Journal for Numerical Methods in Engineering},
  volume={33},
  number={7},
  pages={1331--1364},
  year={1992},
  publisher={Wiley Online Library}
}

@article{ZhangNaga05,
  title={A new finite element gradient recovery method: superconvergence property},
  author={Zhang, Zhimin and Naga, Ahmed},
  journal={SIAM J. Sci. Comput.},
  volume={26},
  pages={1192--1213},
  year={2005},
  publisher={SIAM}
}

@article{NagaZhang04,
  title={A posteriori error estimates based on the polynomial preserving recovery},
  author={Naga, Ahmed and Zhang, Zhimin},
  journal={SIAM J. Numer. Anal.},
  volume={42},
  pages={1780--1800},
  year={2004},
  publisher={SIAM}
}

@article{GuoZhangZou18,
  title={A {$C^0$} linear finite element method for biharmonic problems},
  author={Guo, Hailong and Zhang, Zhimin and Zou, Qingsong},
  journal={Journal of Scientific Computing},
  volume={74},
  number={3},
  pages={1397--1422},
  year={2018},
  publisher={Springer}
}

@article{GuoYang17,
  title={Polynomial preserving recovery for high frequency wave propagation},
  author={Guo, Hailong and Yang, Xu},
  journal={Journal of Scientific Computing},
  volume={71},
  number={2},
  pages={594--614},
  year={2017},
  publisher={Springer}
}

@article{TangWangYuanZhou20,
  title={Rational spectral methods for {PDE}s involving fractional {L}aplacian in unbounded domains},
  author={Tang, Tao and Wang, Li-Lian and Yuan, Huifang and Zhou, Tao},
  journal={SIAM J. Sci. Comput.},
  volume={42},
  number={2},
  pages={A585--A611},
  year={2020},
  publisher={SIAM}
}

@article{Rossmanith06,
  title={An unstaggered, high-resolution constrained transport method for magnetohydrodynamic flows},
  author={Rossmanith, James A},
  journal={SIAM J. Sci. Comput.},
  volume={28},
  number={5},
  pages={1766--1797},
  year={2006},
  publisher={SIAM}
}

@article{hu2025well,
  title={Well-posedness of the relaxed Electron MHD equations with random diffusion},
  author={Hu, Ruimeng and Peng, Qirui and Yang, Xu},
  journal={arXiv:2509.18640},
  year={2025}
}
